\documentclass[11pt]{article}
\usepackage{graphics}
\usepackage{lineno,hyperref}
\usepackage{amsfonts}
\usepackage{amsmath}
\usepackage{amsthm}

\newtheorem{theo}{Theorem}[section]
\newtheorem{lemma}{Lemma}[section]
\newtheorem{corol}{Corollary}[section]
\newtheorem{prop}{Proposition}[section]
\newtheorem{es}{Example}[section]
\newtheorem{rem}{Remark}[section]
\def\RR{\vbox {\hbox{I\hskip-2.1pt R\hfil}}}

\def\PP{\vbox {\hbox{I\hskip-2.1pt P\hfil}}}
\def\C{\mathcal{C}}

\newcommand{\ga}{\gamma}

\newcommand{\de}{\delta}

\newcommand{\be}{\beta}
\newcommand{\la}{\lambda}
\newcommand{\al}{\alpha}
\newcommand{\vphi}{\varphi}
\newcommand{\vgd}{u}

\newcommand{\vgdf}{u\vphi}
\newcommand{\vgdfy}{(u\vphi)(y)}
\newcommand{\vv}{u \rho}
\newcommand{\vgdal}{u \rho }

\def\v{\rho}
\def\iv{\rho^{-1}}
\newcommand{\nequiv}{{\ \equiv{\hspace{-.34cm}\slash} \ }}

\newcommand{\sgn}{{\rm sgn}}
\newcommand{\diag}{{\rm diag}}
\newcommand{\cond}{{\rm cond}}
\newcommand{\beqns}{\begin{eqnarray*}}
\newcommand{\beqn}{\begin{eqnarray}}
\newcommand{\beq}{\begin{equation}}
\newcommand{\eeqns}{\end{eqnarray*}}
\newcommand{\eeqn}{\end{eqnarray}}
\newcommand{\eeq}{\end{equation}}

\begin{document}


\title{Quadrature methods for
integro-differential equations of Prandtl's type in weighted spaces of continuous functions}

\author{Maria Carmela De Bonis and Donatella Occorsio}

%
\maketitle

\begin{abstract}
The paper deals with the approximate solution of integro-differential equations of Prandtl's type.
Quadrature methods involving ``optimal'' Lagrange interpolation processes are proposed and conditions under which they are stable and convergent in suitable weighted spaces of continuous functions are proved.

The efficiency of the method has been tested by  some numerical experiments,  some of them including comparisons with other numerical procedures. In particular,  as an application, we have implemented the method for solving Prandtl's equation governing the circulation air flow along the contour of a plane wing profile, in the case of elliptic or rectangular wing-shape.
\end{abstract}




\section{Introduction}
Hypersingular Integro-Differential Equations (IDE) find application in the treatment of
  many  physics and engineering problems  (for instance, see \cite{vainikko}, \cite{muske}, \cite{mkhi}, \cite{aruty}, \cite{smirnov} \cite{MoPe} and the references therein).
In particular, the IDE of Prandtl's type
\begin{equation}\label{iniziale1}
\sigma(y)\zeta(y)+a \zeta'(y)+\frac b\pi\int_{-1}^1\frac{\zeta'(x)}{x-y}dx+\frac 1\pi\int_{-1}^1 \bar k(x,y)\zeta(x)dx=g(y), \  y\in (-1, 1),
\end{equation}
with $\sigma(y), \bar k(x,y)$ and $g(y)$ given functions, the constants $a, b\in \RR$  s.t. $a^2+b^2=1$, and the unknown solution $\zeta$ is a differentiable function,
satisfying the zero boundary condition
\begin{equation}\label{ucond}\zeta(-1)=\zeta(1)=0,\end{equation}
is well-known  in aerodynamics. In fact, the solution $\zeta$ can represent the circulation distribution of air flow along the contour of a wing profile (see, for instance, \cite{Pra}, \cite{mkhi}, \cite{dragos1}, \cite{dragos2}, \cite{Gol} and the references therein). (Some experiments concerned with this application will be proposed in Section 4.)

Taking into account the zero boundary condition \eqref{ucond}, the solution  $\zeta$ is conveniently represented as the product of a smooth function $f$ for a Jacobi weight, i.e.
\begin{equation}\label{zetaesp}\zeta(x)=f(x)v^{\al,\be}(x),\ \ v^{\al,\be}(x)=(1-x)^\al (1+x)^{\be}, \ \al,\be> 0.\end{equation}

Several authors have studied this kind of IDEs and  introduced numerical methods for approximating their solutions (see \cite{vekua,kalandiya, golberg, CaCriJu,Calio, CaCriJuLu} and the references therein), mainly  in the  case $\al=\be=\frac 12$.
In \cite{CaCriJu} and \cite{CaCriJuLu}, when $\sigma\equiv 0$, the equation has been also considered in the more general case $0<\al<1,\ \be=1-\al$. In particular in  \cite{CaCriJu} the authors introduced collocation and quadrature methods based on Jacobi zeros studying  stability and convergence in weighted $L^2$ spaces and in \cite{CaCriJuLu} a regularized version of \eqref{iniziale1} has been investigated in a scale of pairs of weighted  Besov spaces.

Here, we consider the equation \eqref{iniziale1} both  for $\sigma \nequiv \ 0$, $\al=\be=\frac 12$ and for $\sigma\equiv 0$, $0<\al<1,\ \be=1-\al$. In both  cases we seek the solution in a couple of weighted Zygmund-type spaces equipped with uniform norm. Two quadrature methods which make use of optimal Lagrange interpolation processes are proposed  and for them we determine conditions assuring stability and convergence. The error estimates in weighted uniform norm and the conditioning  of the final linear systems are studied. Finally, some  numerical tests, which confirm the  agreement among the theoretical estimates with the numerical results, are provided.

The plan of the paper is the following. Next section contains some basic results and notation used along the paper. In Section 3 the numerical procedures are described  and the results about their stability and convergence are stated. Section 4 contains some numerical tests to show the efficiency  of the proposed  procedure, some of them in comparison with other ones. In Section 5 the proofs of the main results are given, while Section 6  contains conclusions and a brief discussion on the numerical experiments.

\section{Preliminaries}
From now on the following setting will be used along all the paper:
\begin{equation}\label{defpesi}u=v^{\ga,\de}, \ga, \de\ge 0,\quad w=v^{1-\al,\al},\quad \rho=v^{\al, 1-\al},\quad 0<\al<1.\end{equation}
Moreover the constant $\C$ will be used several times, having different meaning in different formulas.  We will write $\C \neq \mathcal{C}(a,b,\ldots)$  to say that $\C$ is a positive constant independent of the parameters $a,b,\ldots$, and $\C = \C(a,b,\ldots)$ to say that
$\mathcal{C}$ depends on $a,b,\ldots$. If $A,B \geq 0$ are quantities depending on some parameters, we will write $A \sim B,$ if there exists a constant $0<\C\neq \C(A,B)$  such that
\[\frac{B}{\C} \leq A  \leq \C B. \]
$\PP_m$ will denote the space of the algebraic polynomials of degree at most $m$. For a bivariate function $k(x,y)$ we  use $k_x$ (or $k_y$) to regard  $k$ as  function of the only variable $y$ (or $x$).

Many properties holding for FP integrals  can be found in \cite{kutt}, \cite{monegato}  (see also \cite{DBOLagos} and the references therein). Here we recall \cite[Lemma 6.1, Cap II]{miklin}

\[\frac{d} {dy}\int_{-1}^1\frac{g(x)}{x-y}dx=\int_{-1}^1\frac{g'(x)}{x-y}dx-\frac{g(-1)}{1+y}-\frac{g(1)}{1-y},\quad -1<y<1,\]
holding if $g$ has a  generalized derivative $g'\in L_p(-1,1),$ for some $p>1$.
Then, under the zero endpoints  conditions \eqref{ucond}-\eqref{zetaesp} with $\be=1-\al$, equation \eqref{iniziale1} can be rewritten as
\[(M_{\sigma\v}+DA^{\rho}+K+H)f(y)=g(y),\]
where
\[(M_{\sigma\v} f)(y)= (\sigma\v f)(y), \quad (D q)(y)=\frac{d} {dy}q(y),\]
\[ (A^{\rho}f)(y)=a (f\v)(y)+\frac b\pi\int_{-1}^1\frac{(f\v)(x)}{x-y}dx,\]
\[ (Kf)(y)=\frac 1{\pi}\int_{-1}^1\!\! k(x,y)(f\v)(x)dx, \quad (Hf)(y)=\frac 1{\pi}\int_{-1}^1\!\! h(x,y)(f\v)(x)dx, \]
with $\sigma(y)$ a given function, $k(x,y)$ and $h(x,y)$ smooth and weakly singular kernels, respectively, such that $\bar k $ in (\ref{iniziale1})
satisfies $\bar k(x,y)=k(x,y)+h(x,y).$
The  Fredholm index of the Cauchy singular integral operator $A^{\rho}:L^2_\rho\to L^2_\rho$ is equal to $-1$ if $a=\cos{(\pi \al)}, \ b=-\sin{(\pi\al)}$ (see, for instance, \cite{PSbook}). Here, $L^2_\rho$ is the Hilbert space defined by the inner product
\begin{equation}
\label{inner}<f,g>_\rho=\int_{-1}^1f(x)\overline{g(x)}\rho(x)dx.
\end{equation}

\subsection{Function spaces}
We consider the space of functions
\[
C_{u} = \begin{cases}\left\{f\in C^0 ((-1,1)) \ : \ \lim_{x\to \pm 1^\mp}(f u)(x)=0\right\}, &  \ga>0, \de>0\\
\left\{f\in C^0 ((-1,1]) \ : \ \lim_{x\to -1^+}(f u)(x)=0\right\}, & \ga=0, \de>0\\
\left\{f\in C^0 ([-1,1)) \ : \ \lim_{x\to 1^-}(f u)(x)=0\right\}, & \ga>0, \de=0\\
C^0 ([-1,1]), & \ga=\de=0
\end{cases},
\]
equipped  with the norm
\[
    \|f\|_{C_{u}}:=  \|fu\|_\infty=\max_{|x|\le 1}\left|(fu)(x) \right|.
\]
Somewhere, for brevity, we will set $\|f\|_A:=\max_{x\in A}|f(x)|$.

Note that the limit conditions are necessary for the validity of the Weierstrass theorem in $C_u$.
Then, denoting by
$$E_m(f)_{u}=\inf_{P_m\in \PP_m}\|(f-P_m)u\|_\infty$$
the error of best polynomial approximation of $f\in C_{u}$ by means of polynomials of degree at most $m$, we have \cite[p. 172 (2.5.23)]{mastromilo}
\begin{equation}
\label{weie}\lim_m E_m(f)_u=0.
\end{equation}

Setting $\vphi(x)=\sqrt{1-x^2}$, for any $f\in C_{u}$  and for an integer $k\ge 1$, we consider the main part of the \textit{$\varphi$-modulus of smoothness} \cite[p. 90]{DT}
\begin{equation}
\label{mainMod}\Omega_{\vphi}^k(f,t)_{u}=\sup_{0<\tau\leq t}\|u \Delta_{\tau\vphi}^k f\|_{I_{k\tau}}, \quad I_{k\tau}=[-1+(2k\tau)^2,1-(2k\tau)^2],
\end{equation}
where
\[\Delta_{\tau\vphi}^k f(x)=\sum_{i=0}^k (-1)^i \binom k i f\left(x+\frac{\tau\vphi(x)}2(k-2i)\right).\]

By means of $\Omega_{\vphi}^k(f,t)_{u}$, we define the Zygmund
space of order $r \in \RR, r > 0,$
\[Z_{r,k}(u)=\left\{f\in C_{u} \ : \ \sup_{t> 0}\frac{\Omega_{\vphi}^k(f,t)_{u}}{t^r} <+\infty\right\}, \quad k\geq r,\]
endowed with the  norm
\begin{eqnarray}
\|f\|_{Z_{r,k}(u)}&=&\|f u\|_\infty +  \sup_{t> 0}\frac{\Omega_{\vphi}^k(f,t)_{u}}{t^r}.\label{norm1}
\end{eqnarray}
The following equivalence holds true (see, for instance, \cite[p. 172]{mastromilo})
\begin{eqnarray}
\sup_{t> 0}\frac{\Omega_{\vphi}^k(f,t)_{u}}{t^r}\sim  
\sup_{i\geq 0}(1+i)^{r} E_i(f)_{u}, \label{norm2}
\end{eqnarray}
where the constants in ``$\sim$'' depends on $r$. Such norms equivalence ensures that the definition of the Zygmund space doesn't depend on $k\geq r$ and therefore we will set $Z_{r}(u):=Z_{r,k}(u).$

When $r$ is a positive integer,  we define the Sobolev space
\[W_r(u)=\left\{f\in C_u: f^{(r-1)}\in AC(-1,1),\quad \|f^{(r)}\varphi^r u\|_\infty<\infty\right\},\]
where $AC(-1,1)$ denotes the set of the functions which are absolutely continuous on every closed subinterval of $(-1,1)$, equipped with the  norm
\[\|f\|_{W_r(u)}=\|f u\|_\infty + \|f^{(r)}\vphi^r u\|_\infty.\]


In order to estimate $E_m(f)_u$ we recall the Favard inequality (see, for instance, \cite[p. 172]{mastromilo})
\begin{equation}
\label{Favard}E_m(f)_u\leq \frac{\C}{m^r} \|f\|_{Z_r(u)}, \quad \forall f\in Z_r(u),
\end{equation}
where the constant $\C$ does not depend on $m$ and $f$ but depends on $r$.
Moreover, letting $E_m(f)_{Z_r(u)}=\inf_{P_m\in \PP_m}\|f-P_m\|_{Z_r(u)},$
we recall  \cite[p. 33]{dbmastroparma}
\begin{equation}
\label{Em-est-Zr}E_m(f)_{Z_r(u)}\leq \C \sup_{k\ge 1} k^r E_k(f)_{u}, \quad \C\neq\C(m,f).
\end{equation}
In the sequel we will write $Z_r:=Z_r(v^{0,0})$, $E_m(f)_{v^{0,0}}:=E_m(f)$, and $Z_0(u)= W_0(u)=C_u.$

\subsection{Lagrange interpolation}

For a given Jacobi weight  $\theta=v^{\alpha,\beta}, \ \alpha,\beta>-1$,  let $\{p_m^{\theta}\}_{m=0}^\infty$ be the corresponding sequence of  orthonormal polynomials with positive leading coefficients and let $\{\la_{m,k}^{\theta}\}_{k=1}^m $ be  the Christoffel numbers w.r.t. $\theta$.
Let  $\rho$ and $w$ be defined in (\ref{defpesi}). Let $L^{w}_m(G,x)$ be the Lagrange polynomial interpolating a given function $G\in C_{u\varphi}$ at
the zeros $\{x_i\}_{i=1}^m$ of  $p_m^{w}$ and let $L^{\rho}_m(G,x)$ be the Lagrange polynomial interpolating  $G\in C_{u\rho}$ at the zeros $\{t_i\}_{i=1}^m$ of $p_m^\rho.$ Following an idea in \cite{La,DBSIam}, we represent $L^{w}_m(G,x)$ in the basis
$$\psi_i^{w}(x)=\frac{\la_{m,i}^{w}\sum_{j=0}^{m-1}p_j^{w}(x_i)p_j^w(x)}{(u\varphi)(x_i)},\quad i=1,2,\dots,m,$$
of $\PP_{m-1}$ and  $L^{\rho}_m(G,x)$ in the basis
$$\psi_i^{\rho}(x)=\frac{\la_{m,i}^{\rho}\sum_{j=0}^{m-1}p_j^{\rho}(t_i)p_j^\rho(x)}{(u\rho)(t_i)},\quad i=1,2,\dots,m,$$
of $\PP_{m-1}$. More precisely, we write
\begin{equation}\label{lag_w}L^{w}_m(G,x)=\sum_{i=1}^m \psi_i^{w}(x) (u\varphi G )(x_i)\end{equation}
and
\begin{equation}\label{lag_rho}L^{\rho}_m(G,x)=\sum_{i=1}^m \psi_i^{\rho}(x) (u \rho G )(t_i).\end{equation}
The choice of these bases is crucial in the study of the conditioning of the linear systems involved in our numerical methods (see Theorems \ref{condiz} and \ref{condiz2}).

Next lemma, a consequence of \cite[Theorem 2.2]{MR}, states the conditions under which the above introduced Lagrange processes are optimal:
\begin{lemma}\label{Lag1} Let $0< \al<1$. If $\ga,\de$ satisfy
\begin{equation}
\label{gade-cond3}
 -\frac \al 2 +\frac 14\leq \ga <  -\frac{\al}2+\frac 54 ,\quad\qquad\quad \frac \al 2 -\frac 14\leq \de < \frac{\al}2+\frac 34,
\end{equation}
then
\begin{equation}
\label{lag1alal}\| L_m^{w}(f)\vgdf\|_{\infty}\leq \C \log m \| f \vgdf\|_{\infty}, \quad \forall f\in C_{\vgdf},
\end{equation}
\begin{equation}
\label{lagal1al}\| L_m^{\rho}(f)u\rho\|_{\infty}\leq \C \log m \| f u\rho \|_{\infty}, \quad \forall f\in C_{u\rho},
\end{equation}
where $\C\neq \C(m,f).$
\end{lemma}

The following lemma, a special case of \cite[Th.1, p. 680]{Nevai}, will be also useful in the sequel.
\begin{lemma}\label{Nevai} Let $0< \al<1$. If $\ga,\de$ satisfy $0\leq  \ga <  -\frac{\al}2+\frac 34, \ 0\leq \de < \frac{\al}2+\frac 14,$
then, for any $f\in C^0([-1, 1])$,
\[\int_{-1}^1 |L_m^{\rho}(f,x)|u^{-1}(x)dx \leq \C \|f\|_{\infty},\quad \C\neq \C(m,f).\]

\end{lemma}

\section{Main results}
We present now our main results, concerned with the equations
\begin{eqnarray}(DA^{\rho}+K+H)f&=&g, \label{eqSigma0}\quad 0<\al<1,\\
(M_{\sigma\vphi}+DA^{\vphi}+K+H)f&=&g\label{eqSigma}.
\end{eqnarray}

\noindent We start  investigating  (\ref{eqSigma0}) in the pair  of  Zygmund spaces $\left(Z_r(\vgdal), Z_{r-1}(\vgdf)\right).$
\begin{theo}\label{teo-sigma0} Let $0<\al<1$. Assume that
with $\ga,\de$ satisfying
\begin{equation}\max\left\{0,-\frac \al 2 +\frac 14\right\}\leq \ga <  -\frac{\al}2+\frac 12,\qquad
\max\left\{0,\frac \al 2 -\frac 14\right\}\leq \de < \frac{\al}2,      \label{ipotesitheo3.2}
\end{equation}
 and for some $s>0$ it is
\begin{equation}\label{ky-cond}
\sup_{|x|\leq 1}\|k_x\|_{Z_{s}(\vgdf)}<+\infty,
\end{equation}
\begin{equation}\label{h-cond1}\sup_{|y|\leq 1}\vgdfy  \int_{-1}^1 |h(x,y)|u^{-1}(x)dx <+\infty,
\end{equation}
and
\begin{equation}\label{h-cond2}
A(\tau):=\sup_{y\in I_\tau}(u\varphi)(y) \int_{-1}^1 |\Delta _{\tau\vphi(y)}h(x,y)|u^{-1}(x)dx<\C \tau^s,
\end{equation}
with $\C\neq\C(\tau)$, $I_\tau=[-1+(2\tau)^2,1-(2\tau)^2]$. If $Ker(DA^{\rho}+K+H)=\{0\}$ in $Z_r(\vgdal)$ with $1< r<s+1$, then  equation (\ref{eqSigma0})
admits a unique solution $f^*$ in $Z_r(\vgdal)$ for any $g\in Z_{r-1}(\vgdf)$.
\end{theo}

Provided the conditions assuring existence and uniqueness of the solution of equation \eqref{eqSigma0},  we go to describe the
 numerical method proposed to approximate its solution.
Letting  
\begin{equation}\label{Km-def}(K_mf)(y)=\frac 1{\pi}\int_{-1}^1 L_m^{\rho}(k_y,x)(f\v)(x)dx,\end{equation}
we proceed to solve the finite dimensional equation
\begin{equation}\label{eqap}(D A^{\rho}+L_m^{w} K_m+L_m^{w}H)f_m=L_m^{w}g, \quad m\geq 1,\end{equation}
in the unknown $f_m$, where 
 \begin{equation}\label{fm1}f_m(y)=\sum_{k=1}^m \psi_k^{\rho}(y)a_k.  \end{equation}
 Since by \cite[Theorems 9.9 and 9.14, Remark 9.15]{PSbook} and \cite[(4.21.7)]{Szego} we get
\begin{eqnarray}
\label{Dap}DA^{\rho} p_m^{\rho}= (m+1)p_m^{w}, \quad m=0,1,\ldots,
\end{eqnarray}
equation \eqref{eqap} can be written as
\begin{equation*}
L_m^{w}( D A^{\rho}f_m +K_m f_m+H f_m)=L_m^{w}(g)
\end{equation*}
and collocating it at the zeros $\{x_i:= x_i^{w}\}_{i=1}^m$ of $p_m^{w}$, we get, for $i=1,\dots,m,$
\begin{eqnarray}\label{systemdef} (u\vphi D A^{\rho}f_m)(x_i)+ (u\vphi K_mf_m)(x_i)+(u\vphi H f_m)(x_i)=(u\vphi g )(x_i).
\end{eqnarray}
In view of \eqref{lag_rho} and \eqref{Dap}
\begin{equation}
\label{DAxi}(D A^{\rho}f_m)(x_i)=\sum_{k=1}^m \frac{a_k}{(u\rho)(t_k)}\la_{m,k}^{\rho}\sum_{j=0}^{m-1}p_j^{\rho}(t_k)(j+1)
p_j^{w}(x_i)
\end{equation}
and, by \eqref{Km-def},
\begin{equation}
\label{Kxi}(K_mf_m)(x_i)=\frac 1\pi\sum_{k=1}^m \frac{a_k}{(u\rho)(t_k)}\la_{m,k}^{\rho} \, k(t_k,x_i).
\end{equation}
Moreover, we have
\begin{equation}
\label{Hxi} (H f_m)(x_i)=\!\frac 1\pi\!\sum_{k=1}^m \!\frac{a_k \la_{m,k}^{\rho}}
{(u \rho)(t_k)}\!\sum_{j=0}^{m-1}p_j^{\rho}(t_k)c_j(x_i),\ \ \ c_j(y)=\int_{-1}^1\!\! h(x,y)p_j^{\rho}(x)\v(x)dx.
\end{equation}
Thus, combining \eqref{DAxi}, \eqref{Kxi} and \eqref{Hxi} with \eqref{systemdef}, setting $\mathbf{a}_m=[a_1,\dots,a_m]^T,$ we get the linear system
\begin{equation}
\label{system1}\mathbf{A}_m\mathbf{a}_m=\mathbf{b}_m,
\end{equation}
\begin{equation}\label{Ambmdef}\mathbf{A}_m=\mathbf{U}_m\left(\mathbf{V}_m \left[\mathbf{D}_m
\mathbf{Z}_m+\mathbf{W}_m\right]+\mathbf{K}_m\right)\mathbf{\Lambda}_m, \quad \quad \mathbf{b}_m=\mathbf{U}_m \mathbf{g}_m,\end{equation}
with $\displaystyle \mathbf{U}_m=\diag\left((u\varphi)(x_1),\dots,(u\varphi)(x_m) \right),\quad \mathbf{g}_m=[g(x_1),\dots,g(x_m)]^T,$
$$\hspace*{-1cm}\left\{\mathbf{W}_m(i,j)=\frac 1\pi c_i(x_j)\right\}_{_{j=1,\dots,m}^{i=0,1,\dots,m-1}}, \  \left\{ \mathbf{K}_m(i,k)=\frac 1\pi k(t_k,x_i)\right\}_{^{i=1,\dots,m}_{k=1,\dots,m}}, $$
$$\mathbf{D}_m=\diag(1,\dots,m),\quad  \mathbf{\Lambda}_m=\diag\left(\frac{\la_{m,1}^{\rho}}{(\vgdal)(t_1)},\dots,\frac{\la_{m,m}^{\rho}}{(\vgdal)(t_m)} \right),$$
$$\{\mathbf{V}_m(i,j)=p_j^{\rho}(t_i)\}_{^{i=1,\dots,m}_{j=0,1,\dots,m-1}},\quad  \{\mathbf{Z}_m(i,j)=p_i^{w}(x_j)\}_{_{j=1,\dots,m}^{i=0,1,\dots,m-1}}.$$
Therefore, if $\mathbf{a}_m^*=[a_1^*,\dots,a_m^*]^T$ is the unique solution of the linear system \eqref{system1}, we construct  the unique solution of the equation \eqref{eqap} as follows
\begin{equation*}f_m^*(y)=\sum_{k=1}^m \psi_k^{\rho}(y) \ a_k^*.
\end{equation*}

\noindent About the stability and the convergence of the method, we  prove the following
\begin{theo}\label{stability-convergence} Let $0<\al<1$. Let us assume that \eqref{ipotesitheo3.2} holds
and that, for some $s>0$, the kernels $k$ and $h$ satisfy the assumptions \eqref{ky-cond}-\eqref{h-cond2}, $g\in Z_{s}(\vgdf)$, and $Ker(DA^{\rho}+K+H)=\{0\}$ in $Z_r(\vgdal)$ with $1< r<s+1$.\\
Then, for $m$ sufficiently large (say $m>m_0$), the operators $D A^{\rho}+L_m^{w} K_m+L_m^{w}H: (\PP_{m-1},\|\cdot\|_{Z_r(\vgdal)})\to (\PP_{m-1}, \|\cdot\|_{Z_{r-1}(\vgdf)})$ are invertible and their inverses are uniformly bounded.
Moreover, the unique solution $f^*$ of \eqref{eqSigma0} belongs to $Z_{s+1}(\vgdal)$ and if $f_m^*$ denotes the unique solution of \eqref{eqap}, for all $1< r< s+1,$ the following error estimate holds true
\vspace*{-3mm}
\begin{equation}\label{convergence}\|f^*-f_m^*\|_{Z_r(\vgdal)}\leq \C \left(\frac{\log m}{m^{s-r+1}}\right) \|f^*\|_{Z_{s+1}(\vgdal)},\end{equation}
where the constant $\C$ is independent of $m$ and $f^*$. \end{theo}

We conclude with the study of the  linear system conditioning.
\begin{theo}\label{condiz}
Under \ the \ assumptions \ of \ Theorem \ref{stability-convergence}, \ denoting by \ $\cond(\mathbf{A}_m)$ the con\-di\-tion num\-ber of $\mathbf{A}_m$ in infinity norm, we have

\begin{equation}
\label{cond1}\cond(\mathbf{A}_m)\leq  \C \ \left\|\left((DA^{\rho}+L_m^{w}K_m+L_m^{w}H)_{\ |\PP_{m-1}}\right)^{-1}\right\|_{C_{\vgdf}\to C_{\vv}} m \log^3 m,
\end{equation}
where $\C\neq\C(m).$
\end{theo}

Now we treat the case of the equation (\ref{eqSigma}).
Next theorem assigns sufficient conditions under which it is unisolvent.
\begin{theo}\label{teo-sigma}
Let us assume that for some $s>0$, $0\leq \ga <  \frac 14$ and $0\leq \de < \frac{1}4$,
the kernels $k$ and $h$ satisfy the assumptions \eqref{ky-cond}-\eqref{h-cond2} with $\al=\frac 12$ and $\sigma\vphi\in Z_s.$
If $Ker(M_{\sigma\vphi}+DA^{\vphi}+K+H)=\{0\}$ in $Z_r(\vgdf)$ with $1< r<s+1$, then equation \eqref{eqSigma}
admits a unique solution $f^*$ in $Z_r(\vgdf)$ for any $g\in Z_{r-1}(\vgdf)$.
\end{theo}

Now, to approximate the solution of equation \eqref{eqSigma} we  solve the following finite dimensional equation
\begin{equation}\label{eqfinSigma}(L_m^{\vphi}M_{\sigma\vphi}+D A^{\vphi}+L_m^{\vphi}K_m+ L_m^{\vphi}H) f_m=L_m^{\vphi}(g), \quad m\geq 1,\end{equation}
in the unknown
\begin{equation}\label{fm2}f_m(y)=\sum_{k=1}^m \psi_k^{\vphi}(y)\bar a_k, \quad \psi_k^{\vphi}(y)=\frac{\ell_{m,k}^{\vphi}(y)}{(u\varphi)(x_k^{\vphi})}, \quad x_k^{\vphi} \mbox{ zeros of } p_m^{\vphi}.
\end{equation}
By \eqref{DAxi}, \eqref{Kxi}, \eqref{Hxi} with $\al=\frac 12$ and
\[(\vgdf M_{\sigma\vphi}f_m)(x_i^{\vphi})=(\sigma\vphi)(x_i^{\vphi})a_i,\quad 1\le i\le m,\]
the finite dimensional equation \eqref{eqfinSigma} is equivalent to the linear system
\begin{equation*}
\bar{\mathbf{A}}_m \mathbf{\bar a}_m=\mathbf{b}_m,
\end{equation*}
where
$$\mathbf{\bar{a}}_m=[\bar a_1,\dots,\bar a_m]^T, \quad \mathbf{b}_m=\mathbf{U}_m \mathbf{g}_m, \quad \bar{\mathbf{A}}_m=\mathbf{\Gamma}_m+\mathbf{A}_m,$$
with $\mathbf{A}_m$ and $\mathbf{b}_m$ defined in \eqref{Ambmdef} and
$\mathbf{\Gamma}_m=\diag((\sigma\vphi)(x_1^{\vphi}),\dots,(\sigma\vphi)(x_m^{\vphi})).$

About  the stability and the convergence of the method and the conditioning of the linear systems, next theorems hold true.

\begin{theo}\label{stability-convergence2} Under the same assumptions of Theorem \ref{stability-convergence} with $\al=\frac 12$, if for some $s>0$, $\sigma\vphi\in Z_s$ and $Ker(M_{\sigma\vphi}+DA^{\vphi}+K+H)=\{0\}$ in $Z_r(\vgdf),$ with $1< r<s+1,$
then, for $m$ sufficiently large, the operators $L_m^{\vphi}M_{\sigma\vphi} +D A^{\vphi}+L_m^{\vphi}K_m + L_m^{\vphi}H: (\PP_{m-1},\|\cdot\|_{Z_r(\vgdf)})\to (\PP_{m-1}, \|\cdot\|_{Z_{r-1}(\vgdf)})$ are invertible and their inverses are uniformly bounded.

\noindent Moreover,  the unique solution $f_m^*$ of \eqref{eqfinSigma} converges to the unique solution  $f^*\in Z_{s+1}(\vgdf)$ of \eqref{eqSigma} and, for all $1< r< s+1,$ the following error estimate holds
\begin{equation}
\label{convergence2}
\|f^*-f_m^*\|_{Z_r(\vgdf)}\leq \C \left(\frac{\log m}{m^{s-r+1}}\right) \|f^*\|_{Z_{s+1}(\vgdf)},
\end{equation}
where the constant $\C$ is independent of $m$ and $f^*$.
\end{theo}

\begin{theo}\label{condiz2}
Under \ the \ assumptions \ of \ Theorem \ref{stability-convergence2}, \ denoting by $\cond(\bar{\mathbf{A}}_m)$ the  condition number of the matrix $\bar{\mathbf{A}}_m$  in infinity norm, we get
\begin{small}
\begin{equation}
\label{cond2}\cond(\bar{\mathbf{A}}_m)\leq\! \C \! \left\|\left((L_m^{\vphi}M_{\sigma\vphi}+ DA^{\vphi}+L_m^{\vphi}K_m+L_m^{\vphi}H)_{\ |\PP_{m-1}}\right)^{-1}\right\|_{C_{\vgdf}\to C_{\vgdf}}\! m\log^3 m,
\end{equation}
\end{small}
where $\C\neq\C(m).$
\end{theo}

\begin{rem} Firstly we recall that the following  subspace of $L^2_\rho$
\[L_{\rho}^{2,r+1}=\left\{f\in L^2_{\rho} \ : \ \|f\|_{L_{\rho}^{2,r+1}}:= \left(\sum_{n=0}^\infty (1+n)^{2(r+1)} c_n^2\right)^\frac 12<+\infty\right\},\]
where $\{c_n\}_{n=0}^\infty$ are the Fourier coefficients of $f$ in the  orthonormal system $\{p_n^\rho\}_{n=0}^\infty$ w.r.t. the inner product \eqref{inner}, is embedded in the Zygmund space $Z_r(\vgdal)$ (see \cite{MR2}, \cite{F}), i.e.,
\[\|f\|_{Z_r(\vgdal)}\leq \C \|f\|_{L_{\rho}^{2,r+1}}, \quad \C\neq\C(f).\]
This observation could allow to deduce error estimates in $\|\cdot\|_{Z_r(u\rho)}$ starting from that obtained  in $\|\cdot\|_{L_{\rho}^{2,r+1}}.$ In fact, by using the estimate in \cite[Theorem 3.1]{CaCriJu}, one can  prove 
\[\|f^*-f_m^*\|_{Z_r(\vgdal)}\leq \C \|f^*-f_m^*\|_{L_{\rho}^{2,r+1}}\leq \frac{\C}{m^{s-r}}\|f^*\|_{L_{\rho}^{2,s+1}}, \quad \C\neq\C(m,f^*). \]
However, comparing the latter bound with the one in \eqref{convergence},  it is clear  that a direct estimate in Zygmund norm, let us get a better rate of convergence.
\end{rem}

Now, if $f^*$ is the solution of the equation (\ref{eqSigma0}) (or (\ref{eqSigma})) and $f_m^*$ is the solution of \eqref{eqap} (or \eqref{eqfinSigma}), we denote by $\zeta^*=f^*\v$ the exact solution of the initial Prandtl's equation \eqref{iniziale1} and by $\zeta_m^*:=f_m^*\v$
its  $m-$th approximation. By Theorems \ref{stability-convergence} and \ref{stability-convergence2} we can easily deduce the following
\begin{corol} Under the assumptions of Theorems \ref{stability-convergence} or \ref{stability-convergence2}, for any $\varepsilon> 0$, one has
\begin{equation}\label{stima-estrema}\|(\zeta^*-\zeta_m^*)\vgd\|_{\infty}=\|(f^*-f_m^*)\vgdal\|_{\infty} \leq \C \|f^*-f_m^*\|_{Z_{1+\varepsilon}(\vgdal)}\le  \C \ \frac{\log m}{m^{s-\varepsilon}},\end{equation}
where $\C\neq\C(m).$
\end{corol}

\begin{rem}\label{rem3.2} Estimate \eqref{stima-estrema} will be useful in the practical evaluation of the error in the numerical tests, since the discrete absolute error on the left hand side is what we want in order to deduce the number of the exact digits we can reach.
From \eqref{stima-estrema} we can deduce that the convergence order of the proposed method is at least $s-\varepsilon$.
We recall that for functions belonging to $Z_{s+1}(\vgdal)$ the convergence order of the polynomial of best approximation is $s+1$ (see \eqref{Em-est-Zr}).
\end{rem}

\begin{rem}

As you can see, the estimates of $\cond(\mathbf{A}_m)$ and $\cond(\bar{\mathbf{A}}_m)$ given in Theorems \ref{condiz} and \ref{condiz2} are not complete, since we are not able to state the uniformly boundedness of the norms in \eqref{cond1} and \eqref{cond2}.
Nevertheless, the numerical evidences provided by the numerical tests (see Section 4) encourage us to believe that such norms do not increase w.r.t. $m$.
\end{rem}

We conclude by showing some  weakly singular kernels satisfying  \eqref{h-cond1}-\eqref{h-cond2}.
\begin{prop}\label{remark1} Under the assumptions $0\le \ga,\de<1$ and $-1<\mu<0$, the  kernels
\begin{equation*}
h(x,y)=\begin{cases}
|x-y|^\mu,\\
|x-y|^\mu \sgn(x-y),\\
\log|x-y|,\\
|x-y|^\mu \log|x-y|,\end{cases}
\end{equation*}
satisfy \eqref{h-cond1} and, with  $A(t)$  as  in (\ref{h-cond2}), next estimates hold
\begin{equation}\label{xi-value}A(\tau)\leq \C \begin{cases}
\tau^{\mu+1} & h(x,y)=|x-y|^\mu \\
\tau^{\mu+1} & h(x,y)=|x-y|^\mu \sgn(x-y)\\
\tau \log \tau^{-1} & h(x,y)=\log|x-y|\\
\tau^{\mu+1}\log \tau^{-1} & h(x,y)=|x-y|^\mu\log|x-y|.
\end{cases}
\end{equation}
\end{prop}

\section{Numerical Tests}

Now we show the performance of our methods by some numerical examples, where the exact solution $\zeta$ of \eqref{iniziale1} will be approximated by $\zeta_m:=f_m\rho$, with $f_m$   given in \eqref{fm1} or \eqref{fm2}.
When the $\zeta$ is unknown we will retain the approximation $\zeta_{1024}$ as exact.

In the tables we will report, for each  $m$, the maximum absolute error attained by  $\zeta_m$ at the  grid  points  $y_i=-1+\frac i{100}, i=0,\ldots,200,$ i.e.
\begin{equation}\label{errore_def}Err_m=\max_{i=0,\ldots,200}u(y_i)|\zeta(y_i)-\zeta_m(y_i)|, \quad err_m=\max_{i=0,\ldots,200}u(y_i)|\zeta_{1024}(y_i)-\zeta_m(y_i)|. \end{equation}

In order to make comparisons with other methods existing in the literature, in Example 1 we show the numerical results obtained approximating an IDE considered in \cite{Calio} and in Example 2 we compare our results with those achieved with the method in \cite{CaCriJuLu}.

Moreover, to verify the effectiveness of our theoretical estimates, in Examples 2 and 3 we consider suitable test IDEs and we will report the Estimated Order of Convergence (EOC) for increasing values of $m$, i.e.
\[EOC_m=\frac{\log(err_m/err_{2m})}{\log 2}.\]

According to Theorems \ref{condiz} and \ref{condiz2}, the condition numbers of the linear systems increase with $m$ at least as $m \log^3 m$. Presuming a more general increasing behaviour of the condition numbers of order $m^\nu,$ with $\nu>0$, for the Examples 2 and 3,  we will report for each $m$ the following  estimators of $\nu$
\[\nu(m)=\frac{\log\left(\frac{\cond(\mathbf{A}_{2m})}{\cond(\mathbf{A}_m)}\right)}{\log 2} \quad \mbox{and} \quad \overline{\nu}(m)=\frac{\log\left(\frac{\cond(\bar{\mathbf{A}}_{2m})}{\cond(\bar{\mathbf{A}}_m)}\right)}{\log 2},\]
respectively.

The values $\cond(\mathbf{A}_m)$ and $\cond(\bar{\mathbf{A}}_m)$ are computed using the \texttt{MatLab} function \texttt{cond.m} with parameter $P=inf$.

Finally, in Subsection 4.1 we show how our numerical method can be used to approximate the solutions of some special IDEs of Prandtl's type coming from some problems in aerodynamics.

All the computations were performed in $16-$digits arithmetic.

\begin{es} Let us consider the IDE of Prandtl's type \eqref{eqSigma} with
{\small{$$\sigma(y)\!=\!2,\ \ k(x,y)\!\equiv\!0,\  h=\log|x-y|, $$}}
and $g(x)$ such that the exact solution is $\zeta(x)=\sqrt{1-x^2}f(x)$ with $f(x)=\sqrt{(1-x^2)^3}$.
This equation has been considered in \cite{Calio}. The authors show (see \cite[Table 2]{Calio}) only the approximations of the solution obtained for $n=15$: they get at most 2 exact decimal digits. As one can see inspecting Table \ref{tab:b1}, our results are more satisfactory. In fact with $n=16$ we get 3 exact decimal digits and with $n=512$ we attain $11$ exact decimal digits.

\begin{table}[h!]
\caption{Example  4.1}
\label{tab:b1}
\centering
\begin{center}
\begin{tabular}{|c|l|l|} \hline
$m$  & $\cond(\mathbf{\bar A}_m)$  & $Err_m$\\ \hline
$8 $  & $4.9982e+00 $ &$1.5099e-03 $ \\ \hline
$16$  & $9.0130e+00 $ &$7.0718e-05 $ \\ \hline
$32$  & $1.6870e+01 $ &$1.6872e-06 $ \\ \hline
$64$  & $3.2465e+01 $ &$4.5720e-08 $  \\ \hline
$128$ & $6.3581e+01 $ &$8.8290e-10 $ \\ \hline
$256$ & $1.2576e+02 $ &$2.5805e-11 $ \\ \hline
$512$ & $2.5011e+02 $ &$6.4149e-13 $ \\ \hline
\end{tabular}
\end{center}
\end{table}

The same integral equation has been considered in \cite{Calio} also with $g(x)$ such that the exact solution is $\zeta(x)=\sqrt{1-x^2}f(x)$ with $f(x)=x$. Applying their method with $n=35$ the authors get approximations of the solution with at most 3 exact decimal digits. On the contrary, our method allows us to attain approximations of the solution with the machine precision by solving a linear system of order $n=2$.

\end{es}

\begin{es}
Now we consider the equation \eqref{eqSigma0} with $\al=\frac 14$,
$$k(x,y)=\left|\cos\left(y-\frac \pi 4\right)\right|^\frac 92+|\sin(x)|^\frac 72, \ h(x,y)=\left|x-y\right|^{-\frac 13}, \ g(y)=|y|^{\frac{11}2}.$$
The solution is $\zeta(x)=(1-x)^\frac 14 (1+x)^\frac 34f(x)$, $f$ unknown. Here, choosing $\ga=\frac 18$ and $\de=0$ (according to \eqref{ipotesitheo3.2}), $k$ satisfies \eqref{ky-cond} with $s=\frac 92$, $h$ satisfies \eqref{h-cond2} with $s=\frac 23$ (see Proposition \ref{remark1}) and $g\in Z_{\frac {11}2}(v^{\frac 58,\frac 12})$. Thus, by Theorem \ref{stability-convergence}, $f\in Z_{\frac 53}(v^{\frac 38,\frac 34})$ and, by Remark \ref{rem3.2}, the error behaves at least as $\frac{\log m}{m^{\frac 23-\varepsilon}}$.  This slow convergence is confirmed  inspecting  Table \ref{tab:b2}. In fact, the arithmetic mean of the estimated orders of convergence $EOC_m$ is almost $1.1342$. In this case the estimator $\nu(m)$ shows that $\nu\sim 1.00051$.

Note that, the integrals $c_j$ in \eqref{Hxi} have been computed using the recurrence relation showed in \cite[p. 333]{MRT}.

\begin{table}[h!]
\caption{Example  4.2}
\label{tab:b2}
\begin{center}
\begin{tabular}{|c|l|l|l|l|} \hline
$m$  & $\cond(\mathbf{A}_m)$ & $\nu(m)$ & $err_m$& $EOC_m$\\ \hline
$8$  & $5.5777e+00 $&          & $3.5841e-02 $&   \\ \hline
$16$ & $1.1021e+01 $& $9.82621e-01 $& $2.1644e-02 $& $0.7276$\\ \hline
$32$ & $2.1911e+01 $& $9.91306e-01 $& $9.3647e-03 $& $1.2086$\\ \hline
$64$ & $4.3681e+01 $& $9.95335e-01 $& $4.7068e-03 $& $0.9924$\\ \hline
$128$& $8.7390e+01 $& $1.00044e+00 $& $2.2208e-03 $& $1.0836$\\ \hline
$256$& $1.7485e+02 $& $1.00058e+00 $& $9.5749e-04 $& $1.2137$\\ \hline
$512$& $3.4982e+02 $& $1.00051e+00 $& $3.2044e-04 $& $1.5791$\\ \hline
\end{tabular}
\end{center}
\end{table}

Applying the numerical method proposed in \cite[p. 160]{CaCriJuLu} for the numerical resolution of the above integral equation you get the results presented in Table \ref{tab:b22}. As you can see, since the obtained linear systems have higher condition numbers, no correct digits are achieved for the approximations of the solution.

\begin{table}[h!]
\caption{Example  4.2: Numerical results obtained using the method in \cite{CaCriJuLu}}
\label{tab:b22}
\begin{center}
\begin{tabular}{|c|l|l|} \hline
$m$  & $\cond(\mathbf{A}_m)$  & $err_m$\\ \hline
$8 $  & $2.0128e+01 $ &$6.1062e+01 $ \\ \hline
$16$  & $9.3696e+01 $ &$6.0582e+01 $ \\ \hline
$32$  & $4.8291e+02 $ &$5.9618e+01 $ \\ \hline
$64$  & $2.6155e+03 $ &$5.7693e+01 $  \\ \hline
$128$ & $1.4510e+04 $ &$5.3846e+01 $ \\ \hline
$256$ & $8.1311e+04 $ &$4.6153e+01 $ \\ \hline
$512$ & $4.5778e+05 $ &$3.0769e+01 $ \\ \hline
\end{tabular}
\end{center}
\end{table}

\end{es}

\vspace*{-0.3cm}

\begin{es} Consider the integral equation \eqref{eqSigma} with
{\small{$$(\sigma\vphi)(y)=y^2+1,\ \ k(x,y)=\frac{\cos(x+y)}{(x^2+y^2+20)^2}, \ h\equiv 0, \ g(y)=\left|y+\frac 3{10}\right|^\frac 72+y\sin(y).$$}}
In this case the  solution has the form  $\zeta(x)=\sqrt{1-x^2}f(x),$  $f$  unknown. According to \eqref{ipotesitheo3.2} we take $\ga=\de=0$.
Since  $\sigma\vphi\in Z_{s}(\vphi)$ for any $s>0$, $k$ satisfies \eqref{ky-cond} for any $s>0$ and $g\in Z_{\frac 72}(\vphi)$, by Theorem \ref{stability-convergence2}, $s=\frac 72$ and therefore $f\in Z_{\frac 92}(\vphi).$ So,  according to Remark \ref{rem3.2}, the errors behave at least as $\frac{\log m}{m^{\frac 72-\varepsilon}}$. By  Table \ref{tab:b3}, we can conclude that the theoretical expectations are verified, being the arithmetic mean of the $EOC_m$  $\sim 3.9343$. 
In this case,  we have $\cond(\bar{\mathbf{A}}_m)\sim m $ at most.

\begin{table}[h!]
\caption{Example  4.3}
\label{tab:b3}
\begin{center}
\begin{tabular}{|c|l|l|l|l|} \hline
$m$  & $\cond\left(\bar{\mathbf{A}}_m\right)$ & $\overline{\nu}(m)$ & $err_m$& $EOC_m$\\ \hline
$8$  & $4.9498e+00 $&          & $9.7163e-05 $&   \\ \hline
$16$ & $9.1339e+00 $& $8.8384e-01 $& $5.3368e-06 $& $4.18634 $\\ \hline
$32$ & $1.7478e+01 $& $9.3631e-01 $& $3.1042e-07 $& $4.10367 $\\ \hline
$64$ & $3.4150e+01 $& $9.6627e-01 $& $1.5510e-08 $& $4.32298 $\\ \hline
$128$& $6.7481e+01 $& $9.8259e-01 $& $7.4500e-10 $& $4.37980 $\\ \hline
$256$& $1.3413e+02 $& $9.9114e-01 $& $5.1794e-11 $& $3.84638 $\\ \hline
$512$& $2.6744e+02 $& $9.9552e-01 $& $7.6090e-12 $& $2.76699 $\\ \hline
\end{tabular}
\end{center}
\end{table}
\end{es}

\subsection{An application}
The Prandtl's equation (see \cite{dragos1},\cite{dragos2})
\begin{equation}\label{pran_1}\beta C(z)=\frac{\alpha(z)}2\int_{-b}^b \frac{C(\eta)}{(\eta-z)^2}d\eta+j(z),\end{equation}
with the zero boundary  conditions $C(\pm b)=0$,
governs the (unknown) circulation air flow  $C(y)$ along the contour of a plane wing profile.
The constant $\beta=\sqrt{1-M^2},$ where $M$ the Mach number in undisturbed motion, $j,\alpha$ are  given functions  depending on the geometry of the wing and the solution  $C(z)=\sqrt{b^2-z^2}c(z)$.

\subsubsection{Elliptic  wing} In this case, being $x^2+ z^2/b^2=1$ the ellipsis equation,  $2b$ is the wingspan, 
  $$\alpha(z)=b^{-1}\sqrt{b^2-z^2},\quad j(z)=2\pi \alpha(y) \epsilon,$$
with $\epsilon$  acute angle  between the direction of the relative wind and the chord of the wing (the angle of attack).

Introducing the changes of variable $z=by,\ \eta=bx$, we have the equivalent equation %
\begin{equation}\label{pran_2}\beta \zeta(y)=\frac{\tilde \alpha(y)}{2b}\int_{-1}^1 \frac{\zeta(x)}{(x-y)^2}dx+\tilde j(y),\end{equation}
with $\zeta(y)=\sqrt{1-y^2}d_1(y)$ and therefore
\[\frac{2b\beta}\pi d_1(y)-\frac{1}{\pi}\int_{-1}^1 \frac{d_1(x)}{(x-y)^2}\varphi(x)dx=4b\epsilon.\]
Setting $$\sigma_1(y)=\frac{2b\beta}{\pi \varphi(y)} ,\ \rho(y)=\varphi(y),\ g_1(y)=4b\epsilon$$
we have to solve  
\begin{equation}\label{pran_4}(M_{\displaystyle{{\sigma_1}}\vphi}+DA^{\vphi})d_1(y)=g_1(y).\end{equation}
The exact solution is known in this case $$\tilde C(y)=\sqrt{1-y^2}d_1(y)=\sqrt{1-y^2}\frac{4\epsilon b}{1+\frac{2b\beta}\pi}.$$
We have tested our method selecting $\beta=1, b=10$  and choosing two different values for the angle of attack, $\epsilon=0.1$ and $\epsilon=0.0872$. In both the cases, since the solution belongs to $Z_s(\vphi)$ for any $s>1$,   the machine precision is attained by solving a linear system of order $m=2$.
We point out that in \cite{dragos2}, for the same tests,  by using a discretization  based on a $N-th$ Gauss rule, only two exact digits are achieved   with   $N=50$.

\subsubsection{Rectangular wing}
In this case $2b$ is the length of the rectangular's largest dimension and  $$\alpha(z)=1,\quad j(z)=2\pi \epsilon.$$
By (\ref{pran_2}) with $\zeta(y)=\sqrt{1-y^2}d_2(y)$
\[\frac{2b\beta}\pi d_2(y)\varphi(y)-\frac{1}{\pi}\int_{-1}^1 \frac{d_2(x)}{(x-y)^2}\varphi(x)dx=4b\epsilon,\]
and setting $$\sigma_2(y)=\frac{2b\beta}{\pi} ,\ \rho(y)=\varphi(y),\ g_2(y)=4b\epsilon$$
the equation can be rewritten
\begin{equation}\label{pran_5}(M_{{\sigma_2}\vphi}+DA^{\vphi})d_2(y)=g_2(y).\end{equation}
For this case the exact solution is unknown.

Since $\sigma_2\varphi\in Z_{1},$ according to Remark \ref{rem3.2}, the error behaves as $\mathcal{O}\left(\frac{\log m}{m^{1-\varepsilon}} \right)$. Inspecting Table \ref{tab:b4}, one can see that the numerical results are better than the expected ones, as order.

\vspace*{-0.5cm}
\begin{table}[h!]
\caption{Rectangular wing}
\label{tab:b4}
\centering
\begin{center}
\begin{tabular}{|c|l|l|} \hline
$m$  & $\cond(\mathbf{\bar A}_m)$  & $err_m$\\ \hline
$8 $  & $2.9237e+00 $ &$1.0186e-03 $ \\ \hline
$16$  & $4.8340e+00 $ &$4.5344e-05 $ \\ \hline
$32$  & $8.3045e+00 $ &$1.6127e-06 $ \\ \hline
$64$  & $1.5198e+01 $ &$5.7771e-08 $  \\ \hline
$128$ & $2.8963e+01 $ &$2.1943e-09 $ \\ \hline
$256$ & $5.6503e+01 $ &$8.7550e-11 $ \\ \hline
\end{tabular}
\end{center}
\end{table}

\section{The proofs}

\subsection{Proof of Theorem  \ref{teo-sigma0}}

In order to prove the theorem  we need to study the mapping properties of the operators $D, A^{\rho}, K$ and $H$.
To this end for $0<\al<1$ we consider the following subspace of $Z_r(u)$
\[Z_{r,0}(u):=\left\{f\in Z_r(u) \ : \ \int_{-1}^1 f(x)\rho^{-1}(x)dx=0\right\},\ \|f\|_{Z_{r,0}(u)}:= \|f\|_{Z_{r}(u)}.\]

The following lemma states the boundedness of the operator $D: Z_{r}(\vgd) \to  Z_{r-1}(\vgd\vphi), r>1$.
\begin{lemma}\label{exlemma1}
For $r> 1 $
\begin{equation}
\label{ff'equiv}f' \in  Z_{r-1}(u\vphi) \quad \Leftrightarrow \quad f\in Z_{r}(u),
\end{equation}
and
\begin{equation}\label{normequiv}\|f'\|_{Z_{r-1}(\vgd\vphi)}\leq \C \|f\|_{Z_{r}(\vgd)}, \quad \C\neq \C(f).\end{equation}
In particular,  \eqref{normequiv} is not true for $r=1$.
\end{lemma}

\begin{proof}
(\ref{ff'equiv}) follows  by arguments similar to those  used in \cite[p. 337-338]{Timan}.

Start from
\begin{eqnarray*}
\|f'\|_{Z_{r-1}(u\vphi)}&=&\|f' u\vphi\|_\infty + \sup_{t> 0}\frac{\Omega_{\vphi}^{k}(f',t)_{u\vphi}}{t^{r-1}}.
\end{eqnarray*}
First we prove
\begin{equation}\label{second}\sup_{t> 0}\frac{\Omega_{\vphi}^k(f',t)_{u\vphi}}{t^{r-1}}\leq \C \sup_{t> 0}\frac{\Omega_{\vphi}^{k+1}(f,t)_{u}}{t^{r}}, \quad k>r-1. \end{equation}

Since
$f\in C_u$, by \eqref{weie} there exists a sequence $\{P_m\}_m$ of  best approximation polynomials s.t. the series
$P_m+\sum_{i=0}^\infty(P_{2^{i+1}m}-P_{2^i m})$ converges uniformly in $[-1,1]$ to $f$ in $C_u$.
If we prove that the series
\begin{equation}
\label{serie}\sum_{i=0}^\infty (P_{2^{i+1}m}(x)-P_{2^i m}(x))'(u\vphi)(x)
\end{equation}
uniformly converges $\forall x\in [-1,1]$, then the equality
\[\left(\sum_{i=0}^\infty (P_{2^{i+1}m}-P_{2^i m})\right)'u\vphi=\sum_{i=0}^\infty (P_{2^{i+1}m}-P_{2^i m})'u\vphi\]
holds true and the series $P_m'+\sum_{i=0}^\infty(P_{2^{i+1}m}-P_{2^i m})'$ converges uniformly in $[-1,1]$ to $f'$ in $C_{u\vphi}$.

By the Bernstein  inequality \cite[Th. 8.4.7]{DT}
and the weak-Jackson inequality \cite[Th. 8.2.1]{DT}, we have
\begin{eqnarray*}
&&\|(P_{2^{i+1}m}-P_{2^i m})'u\vphi\|_\infty \leq  \C (2^{i+1}m) \|(P_{2^{i+1}m}-P_{2^i m})u\|_\infty\\ &\leq &\C (2^{i+1}m)E_{2^i m}(f)_u \leq  \C (2^{i+1}m) \int_0^{\frac 1{2^i m}} \frac{\Omega_{\vphi}^{k+1}(f,t)_{u}}{t}dt \\
&\leq &\C \frac 1{(2^{i}m)^{r-1}} \sup_{t> 0}\frac{\Omega_{\vphi}^{k+1}(f,t)_{u}}{t^{r}}, \ \ \C\neq\C(m).
\end{eqnarray*}
Thus, by the assumption on $f$, we have
\begin{eqnarray*}\sum_{i=0}^\infty\|(P_{2^{i+1}m}-P_{2^i m})'u \vphi \|_\infty 
&\leq & \frac{\C}{m^{r-1}} \sup_{t> 0}\frac{\Omega_{\vphi}^{k+1}(f,t)_{u}}{t^{r}}
\end{eqnarray*}
and the series \eqref{serie} uniformly converges $\forall x\in [-1,1]$.
Then
\begin{eqnarray} E_m(f')_{u \vphi} & \leq & \|(f-P_m)' u\vphi\|_\infty \leq  \sum_{i=0}^\infty\|(P_{2^{i+1}m}-P_{2^i m})' u\vphi\|_\infty \nonumber \\ & \leq &  \frac{\C}{m^{r-1}} \sup_{t> 0}\frac{\Omega_{\vphi}^{k+1}(f,t)_{u}}{t^{r}},\label{le3.1.1}
\end{eqnarray}
where $\C\neq\C(m)$. Since, using \eqref{norm2}
\[
\sup_{t> 0}\frac{\Omega_{\vphi}^k(f',t)_{u\vphi}}{t^{r-1}}\leq \C \sup_{m\geq 1}m^{r-1} E_m(f')_{u\vphi},
\]
taking into account (\ref{le3.1.1}), (\ref{second}) follows.

Let $Q_{1}$ be the $1$-degree polynomial of best approximation of $f'\in C_{u\vphi}$. We have
\begin{eqnarray*}
\|f' u\vphi\|_\infty&\leq & E_{1}(f')_{u\vphi}+ \|Q_1 u\vphi\|_\infty.
\end{eqnarray*}
By \eqref{le3.1.1} and $\|Q_1 u\vphi\|_\infty\le \C\|fu\|_{\infty}$,
\begin{eqnarray*}
\|f' u\vphi\|_\infty&\leq & \C \left[\sup_{t>0} \frac{\Omega_{\vphi}^{k+1}(f,t)_{u}}{t^r}+ \|fu\|_\infty\right]\le \C\|f\|_{Z_{r}(u)},\ k>r-1.
\end{eqnarray*}
 (\ref{normequiv}) follows by combining the last estimate  with \eqref{second}.

Finally, for $r=1$ \eqref{normequiv} becomes
\[\|f'u\vphi\|_\infty \leq \|f\|_{Z_1(u)} \]
and the above inequality is not true, being $W_1(u)\subset Z_1(u)$. For example the function $f(x)=x \log |x|, |x|\leq 1,$ belongs to $Z_1(u)$ but does not belong to $W_1(u)$ (see \cite[p. 54]{russoPhD}.
\end{proof}

\begin{lemma}\label{Dcontinuo} Let $\ga,\de\geq 0$ and $r> 1$. The operator $D: Z_{r,0}(\vgd) \to  Z_{r-1}(\vgd\vphi)$  is continuous and invertible. Moreover its inverse is bounded.
\end{lemma}
\begin{proof} Since the continuity of $D$ is a consequence of Lemma \ref{exlemma1} it remains to prove only the invertibility. 
By \eqref{ff'equiv} for any $g\in Z_{r-1}(u \vphi)$ there exists $f\in Z_{r}(u)$ s.t.
$Df=g$, i.e. $D:Z_{r}(u) \to Z_{r-1}(u \vphi)$ is surjective. On the other hand, to any  $f\in Z_{r}(u)$ we can associate the function
$\bar f=f-\frac{\int_{-1}^1 f(x)\rho^{-1}(x)dx}{\int_{-1}^1 \rho^{-1}(x)dx}$
belonging to $Z_{r,0}(u)$, then $D:Z_{r,0}(u) \to Z_{r-1}(u \vphi)$ is surjective too.  Since the injectivity  can be easily proved, it follows that
$D:Z_{r,0}(u) \to Z_{r-1}(u \vphi)$ is  invertible for any  $r> 1$. Moreover, by the open mapping theorem (see, for example, \cite[p. 517]{Atk}), the inverse of $D$ is bounded.
\end{proof}

\vspace*{-1cm}

\begin{eqnarray*}
\textrm{Setting\ }
(A^{\rho^{-1}}f)(x)&=&(\cos\pi\al) \iv(x)f(x)+
\frac{\sin\pi\al}\pi\int_{-1}^1
f(y)\frac{\iv(y)}{y-x}dy,
\end{eqnarray*}
the following result can be found in \cite[Corollary 2.2]{D}.
\begin{lemma}
\label{A-teo} Let $0<\al<1$. Under the assumptions in \eqref{ipotesitheo3.2}
the linear maps $$A^{\rho^{-1}}:Z_{r,0}(u)\ \to \ Z_{r}(u\rho), \quad A^{\rho}:Z_{r}(u\rho)\to Z_{r,0}(u)$$ are both  continuous for $r>0$. Moreover,  $A^{\rho}$ is the inverse of $A^{\rho^{-1}}$ and the following equivalences hold true
\begin{eqnarray}
\|A^{\rho^{-1}}f\|_{Z_{r}(\vv)}&\sim & \|f\|_{Z_{r}(u)}, \quad
\|A^{\rho}f\|_{Z_{r}(u)}\sim  \|f\|_{Z_{r}(\vv)}, \label{normA}
\end{eqnarray}
 where the constants in ``$\sim$" are independent of $f.$
\end{lemma}

As a consequence of Lemmas \ref{Dcontinuo} and \ref{A-teo} we deduce the following result.
\begin{corol}
\label{corolInv} Let $0<\al<1$. Under the assumptions in \eqref{ipotesitheo3.2}
the operator $DA^{\rho}: Z_{r}(\vv) \to Z_{r-1}(u \vphi)$ is continuous and invertible for each $r> 1$. Moreover its inverse is bounded.
\end{corol}

The following lemma will be useful in the sequel.
\begin{lemma}\label{maky_lem}
Let us assume that the kernel $k(x,y)$ satisfies \eqref{ky-cond}. Then there exists a sequence $\{P_m\}_m$ of polynomials $P_m(x,y)=\sum_{i=0}^m p_{i,m}(x)y^{i},$ of degree not greater than $m$ in $y$, such that $p_{i,m}(x)$ is piecewise constant for all
$i =0,\ldots,m$ and
\begin{equation}
\label{maky}\sup_{x,y\in [-1,1]} (u\vphi)(y)|P_m(x,y)-k(x,y)|\leq \C \sup_{|x|\leq 1}E_m(k_x)_{\vgdf},
\end{equation}
where $\C\neq \C(m)$.
\end{lemma}
\begin{proof}
The proof can be easily deduced following step by step the proof of Lemma 4.11 in \cite{JuLu}.
\end{proof}

\begin{lemma}\label{K-compact-1}Let $0<\al<1$ and let  $\ga,\de<1$. If for some $s> 0$ the kernel $k$ satisfies \eqref{ky-cond},
then $K : C_{u\rho} \to  Z_{r-1}(u \vphi) $ is continuous for all  $1\leq r\leq s+1$ and compact for all $1\leq r<s+1$.
\end{lemma}
\begin{proof} Taking into account (\ref{ky-cond}) we have
\begin{eqnarray}\vgdfy|(Kf)(y)|
&\leq & \frac 1\pi\|f\vv\|_\infty \vgdfy\int_{-1}^1 |k(x,y)|u^{-1}(x)dx\nonumber\\
&\leq & \C \|f\|_{C_{\vv}}  \sup_{|x|\leq 1}\|k_x \vgdf\|_\infty \leq  \C \|f\|_{C_{\vv}}. \label{Kbound1}
 \end{eqnarray}
Let  $\{P_{m}\}_m$ be the sequence of polynomials defined in Lemma \ref{maky_lem}. Then
\begin{eqnarray*}
E_m(K f)_{\vgdf}&\leq & \frac 1\pi \sup_{|y|\leq 1}\vgdfy\int_{-1}^1|k(x,y)-P_m(x,y)||(f\v)(x)|dx \\
&\leq & \C \|f\vv\|_\infty \sup_{x, y\in [-1,1]}\vgdfy |k(x,y)-P_m(x,y)|\int_{-1}^1u^{-1}(x)dx \\
&\leq & \C \|f\vv\|_\infty \sup_{|x|\leq 1} E_m(k_x)_{\vgdf}.
\end{eqnarray*}
Since, under the  assumption \eqref{ky-cond}, $k_x\in Z_s(\vgdf),$ using \eqref{Favard}, we get
\begin{equation}\label{EmKf-est-1}E_m(K f)_{\vgdf}\leq \frac{\C}{m^s}\|f\|_{C_{\vv}}.\end{equation}
Combining \eqref{Kbound1} and \eqref{EmKf-est-1} with  \eqref{norm1} and \eqref{norm2},  the continuity of $K: C_{\vv}\to Z_{r-1}(\vgdf)$ with $1\leq r\leq s+1$ follows.

Now, since for all $f\in Z_{r-1}(u\vphi)$ we have \cite[p. 6]{MR}
\begin{equation}
\label{immZ}\|f-L_m^w(f)\|_{Z_{r-1}(u\vphi)}\leq \frac{\C}{m^{s-r+1}}\|f\|_{Z_s(u\vphi)}\log m, \quad 0\leq r-1<s,
\end{equation}
the imbedding operator $E:Z_s(u\vphi)\to Z_{r-1}(u\vphi)$ can be approximated by a sequence of finite dimensional operators and then $Z_s(u\vphi)$ is compactly imbedded in $Z_{r-1}(u\vphi)$ for $0\leq r-1< s$. Consequently, from the continuity of the operator $K :C_{\vv}\to Z_{s}(\vgdf)$ we deduce the compactness of the operator $K :C_{\vv}\to Z_{r-1}(\vgdf)$ for $1\leq r<s+1$.
\end{proof}

\begin{lemma}\label{H-compact-1}Let $0<\al<1$. If the kernel $h(x,y)$ satisfies \eqref{h-cond1} and \eqref{h-cond2},
then the operator $H : C_{\vv} \to  Z_{r-1}(\vgd \vphi), $ is continuous for all  $1\leq r\leq s+1$ and compact for all $1\leq r<s+1$.

In particular, when $\al=\frac 12$ and $h(x,y)=\log|x-y|$, for all $r\geq 1$ the operator $H$ is continuous as a map from $Z_{r}(\varphi)$ into $Z_{r}(\varphi)$ and compact as a map from $Z_{r}(\varphi)$ into $Z_{r-1}(\varphi)$.
\end{lemma}

\begin{proof} We have
\begin{eqnarray}\label{Hbound}\vgdfy|(Hf)(y)|
&\leq & \C \|f\vv\|_\infty \vgdfy\int_{-1}^1 |h(x,y)|u^{-1}(x)dx,
\end{eqnarray}
and, by \eqref{h-cond1}, we deduce the continuity of the operator $H:C_{\vv}\to\C_{\vgdf}$.

We observe that the compactness of $H:C_{\vv} \to  Z_{r-1}(\vgd \vphi)$ can be proved  if the following estimate holds  
\begin{equation}\label{OmegaH}
\frac{\Omega_{\vphi}(Hf,t)_{\vgdf}}{t^s}\leq \C \|f\|_{C_{\vv}}
\end{equation}
for some $s>r-1$.
Indeed, by using the weak-Jackson inequality \cite[Th. 8.2.1]{DT} together with \eqref{OmegaH}, we get
\begin{equation}\label{EmHf}E_m(Hf)_{\vgdf}\leq \frac{\C}{m^s} \|f\|_{C_{\vv}},\end{equation}
and, combining \eqref{Hbound} and \eqref{EmHf} with \eqref{norm1} and \eqref{norm2} we get the con\-ti\-nui\-ty of \ $H:C_{\vv} \ $ $\to  Z_{r-1}(\vgdf)$, $1\leq r\leq s+1$.
Moreover, by (\ref{Em-est-Zr}) and (\ref{EmHf}), we obtain
\begin{eqnarray*}
E_m(H f)_{Z_{r-1}(\vgdf)}&\leq& \C \sup_m m^{r-1} E_m(Hf)_{\vgdf}\leq \frac{\C}{m^{s-r+1}} \|f\|_{C_{\vv}}, \quad s>r-1,  
\end{eqnarray*}
and, therefore,
by \cite[p. 44]{Timan} it follows  that $H :C_{\vv}\to Z_{r-1}(\vgdf)$ is compact. \\
So, it remains to prove \eqref{OmegaH}. By the assumption \eqref{h-cond2}, we have
\begin{eqnarray*}&&\|\vgdf \Delta_{\tau\vphi} Hf\|_{I_{\tau}}= \frac 1\pi\sup_{y\in I_{\tau}} \vgdfy \left|\int_{-1}^1 \Delta_{\tau\vphi(y)}h (x,y) f(x)\v(x)dx\right|\\
&\leq & \C \|f\vv\|_\infty \sup_{y\in I_{\tau}}  \vgdfy\int_{-1}^1 |\Delta_{\tau\vphi(y)} h(x,y)|u^{-1}(x)dx \leq  \C \|f\|_{C_{\vv}}  \tau^s.
\end{eqnarray*}
Thus, by  definition of $\Omega_\vphi$,
\eqref{OmegaH} follows.

The case $\al=\frac 12$ and $h(x,y)=\log|x-y|$ is special since $\frac{d}{dy}(Hf)(y)=-(A^\vphi f)(y)$. The proof can be deduced following \cite[Proof of Theorem 2.2]{MT2}.

\end{proof}

\begin{proof}[Proof of Theorem \ref{teo-sigma0}]
The theorem follows by Corollary \ref{corolInv}, Lemmas \ref{K-compact-1},\ref{H-compact-1} and the Fredholm alternative Theorem (see, for instance, \cite[Cor. 3.8]{Kress}).
\end{proof}

\subsection{Proofs of Theorems \ref{stability-convergence} and \ref{condiz}}
In order to prove the theorems, we need the following lemmas.
\begin{lemma}\label{(K-Km)fm} Let $0< \al<1$. If, for some $s>0$ and $\ga,\de$ satisfying \eqref{ipotesitheo3.2}, the kernel $k$ satisfies \eqref{ky-cond},
then, for every $1\leq r<s+1$,
{\small{\[\|(K-L_m^{w}K_m)f\|_{Z_{r-1}(\vgdf)}\leq \C \|f\|_{C_{\vv}}\sup_{|x|\leq 1}\|k_x\|_{Z_{s}(\vgdf)} \frac{\log m}{m^{s-r+1}}, \ \C\neq \C(m,f,k).\]}}
\end{lemma}
\begin{proof} By definitions of $K$ and $K_m$ we have
{\small{\[(Kf)(y)\!-\!L_m^{w}(K_mf)(y)\!\!=\!\!\frac 1\pi \left[\int_{-1}^1\! \!k(x,y)(f\v)(x)dx-\!L_m^{w}\left(\!\int_{-1}^1 \! L_m^{\rho}(k_y,x)(f\v)(x)dx,y \right)\right].\]}}
In what follows for any $a: [-1,1]^2\to \RR$, we set
{\small{\[(K^af)(y)\!\!=\!\! \frac 1\pi\int_{-1}^1 \!\! a(x,y)(f\v)(x)dx, (\!\tilde K_m^a f)(y)\!=\! \frac 1\pi
 L_m^{w}\left(\int_{-1}^1 L_m^{\rho}\left(\!a(\cdot,y),x\right)(f\v)(x)dx, y\!\! \right).\]}}
Let $\{P_m\}_m$ be the sequence of polynomials defined in Lemma \ref{maky_lem}. Since
$K^{P_m}-\tilde K_m^{P_m}=0$, then
 for $R(x,y)=k(x,y)-P_m(x,y)$ we get
\begin{equation*}\|\vgdf (K-L_m^{w}K_m)f \|_\infty \leq \|\vgdf K^R f \|_\infty +\|\vgdf \tilde K_m^R f \|_\infty.\end{equation*}
By (\ref{Kbound1}), we obtain
\begin{eqnarray*}\|\vgdf K^R f \|_\infty &\leq & \C\|f\|_{C_{\vv}} \sup_{|x|\leq 1} \|R_x \vgdf\|_\infty,\end{eqnarray*}
and, by Lemmas \ref{Lag1} and \ref{Nevai}, we deduce
\begin{eqnarray*}
\|\vgdf \tilde K_m^R f \|_\infty &\leq &\C \log m \ \|f\|_{C_{\vv}} \sup_{|y|\leq 1}\vgdfy\int_{-1}^1 |L_m^{\rho}\left(R_y,x\right)|u^{-1}(x)dx \nonumber \\
&\leq & \C \log m \  \|f\|_{C_{\vv}} \sup_{|y|\leq 1} \vgdfy \|R_y\|_\infty.
\end{eqnarray*}
Thus, by Lemma \ref{maky_lem},
\begin{eqnarray*}
\|\vgdf(K-L_m^{w}K_m)f \|_\infty &\leq&  \C \log m\  \|f\|_{C_{\vv}} \sup_{x, y\in [-1,1]} \vgdfy|k(x,y)-P_m(x,y)| \\
&\leq & \C \log m \  \|f\|_{C_{\vv}} \sup_{|x|\leq 1} E_m(k_x)_{\vgdf}.
\end{eqnarray*}
Finally, in virtue of the assumption \eqref{ky-cond} on $k$, using \eqref{Favard} we obtain
\begin{equation}\label{K-Kmf-inf}\|\vgdf (K-L_m^{w}K_m)f \|_\infty \leq \C  \|f\|_{C_{\vv}} \sup_{|x|\leq 1}\|k_x\|_{Z_{s}(\vgdf)} \frac{\log m}{m^s}, \quad \C=\C(s). \end{equation}
Now, taking into account the equivalence (\ref{norm2}), we get
\begin{eqnarray*}
\|\vgdf(K-L_m^{w}K_m)f \|_{Z_{r-1}(\vgdf)}&\leq &\C \|\vgdf (K-L_m^{w}K_m)f \|_\infty \\
&+& \C \sup_m m^{r-1} \|\vgdf (K-L_m^{w}K_m)f \|_\infty, \, \C=\C(r,s),
\end{eqnarray*}
and  the thesis follows combining last estimate with  (\ref{K-Kmf-inf}).
\end{proof}
\begin{lemma}\label{(H-Hm)fm} Let $0< \al<1$. If, for some $s>0$ and $\ga,\de$ satisfying $(\ref{gade-cond3})$, the kernel $h$ satisfies \eqref{h-cond1}-\eqref{h-cond2},
then, for every $1\leq r<s+1$,
\[\|(H-L_m^{w}H)f\|_{Z_{r-1}(\vgdf)}\leq \C \|f\|_{C_{\vv}} \frac{\log m}{m^{s-r+1}},\quad \C\neq \C(m,f).\]
\end{lemma}
\begin{proof} By Lemma \ref{Lag1} and using \eqref{EmHf} we get
\begin{eqnarray}\label{H-Hmfest}
\|\vgdf(H-L_m^{w}H)f\|_\infty \leq \C \log m E_{m-1}(Hf)_{\vgdf} \leq \frac{\log m}{m^s} \|f\|_{C_{\vv}}.
\end{eqnarray}
Using the equivalence (\ref{norm2}) and (\ref{H-Hmfest}) the theorem follows. 
\end{proof}

\begin{proof}  [ Proof of Theorem \ref{stability-convergence}.]  We \ first \ note \ that, by \ Corollary \ref{corolInv}, \ $g\in Z_s(\vgdf)$ \ implies \ $[DA^{\rho}]^{-1}g\in Z_{s+1}(\vv)$ \ and, \ by \ Lemmas \ \ref{K-compact-1} \ and \ \ref{H-compact-1}, \
$[DA^{\rho}]^{-1}Kf, $ $[DA^{\rho}]^{-1}Hf$ $\in Z_{s+1}(\vv)$ for any $f\in C_{\vv},$ then \\
$$f=-[DA^{\rho}]^{-1}Kf-[DA^{\rho}]^{-1}Hf+[DA^{\rho}]^{-1}g\in Z_{s+1}(\vv)$$
too.

Since by Lemmas \ref{(K-Km)fm} and \ref{(H-Hm)fm} we can choose $m$ sufficiently large (say $m>m_0$) such that
\[\|(DA^{\rho}+K+H)^{-1}[(L_m^{w}K_m-K)+(L_m^{w}H-H)]\|_{Z_{r}(\vgdal)\to Z_r(\vgdal)}<1,\]
then, using a well-known result (see, for example, \cite[Theorem 10.1, p. 142]{Kress}), the inverse operators  $(DA^{\rho}+L_m^{w}K_m+L_m^{w}H)^{-1} : Z_{r-1}(\vgdf)\to Z_r(\vgdal)$ exist and are uniformly bounded w.r.t. $m$, i.e.
\begin{equation}\label{invBound}\sup_{m\geq m_0} \|(DA^{\rho}+L_m^{w}K_m+L_m^{w}H)^{-1}\|_{Z_{r-1}(\vgdf)\to Z_r(\vgdal)}<+\infty.\end{equation}
In order to prove (\ref{convergence}), we use the following identity
\begin{eqnarray}\label{rel2}(f-f_m)=(DA^{\rho}+L_m^{w}K_m+L_m^{w}H)^{-1}\!\!\!& &\left[(g-L_m^{w} g)-(L_m^{w}K_m-K)f\right.\nonumber \\
& &\left. -(L_m^{w}H-H)f\right].
\end{eqnarray}
Since by \eqref{immZ} and the assumptions on $g$ it is
$$\|g-L_m^{w} g\|_{Z_{r-1}(\vgdf)}\leq \C \frac{\log m}{m^{s-r+1}}\|g\|_{Z_s(\vgdf)},$$
using \eqref{invBound} and Lemmas $\ref{(K-Km)fm}$ and \ref{(H-Hm)fm}, we get
\begin{equation}
\label{rel22}\|f-f_m\|_{Z_r(u\rho)}\leq \C \frac{\log m}{m^{s-r+1}}\left[\|g\|_{Z_s(\vgdf)} + \|f\|_{Z_{s+1}(u\rho)}\right].
\end{equation}
On the other hand, since $g=(DA^{\rho}+K+H)^{-1} f$, by Theorem \ref{teo-sigma0},
\begin{equation}\label{gest}
\|g\|_{Z_s(u\vphi)}\leq \C \|f\|_{Z_{s+1}(u\rho)}.
\end{equation}
Combining \eqref{gest} and \eqref{rel22}, (\ref{convergence}) follows.
\end{proof}

\begin{proof}[Proof of Theorem \ref{condiz}]
Let  $P\in\PP_{m-1},$ $m>1$. Using \cite[Th. 3.1]{russoLuther} with $w_1=\vv$ and $w_2=u$, we get
\begin{eqnarray*}\|A^\rho P\|_{C_u}&\leq &\C \|P\|_{C_{\vv}} + \C \int_0^1 \frac{\Omega_\vphi^k(P,t)_{\vv}}t dt\\ &=&
\|P\|_{C_{\vv}} + \C \left\{\int_0^{\frac 1 m}+ \int_{\frac 1 m}^1\right\} \frac{\Omega_\vphi^k(P,t)_{\vv}}t dt.\end{eqnarray*}
Applying
$\Omega_\vphi^k(P,t)_{\vv}\le \C t \|P'\varphi \vv\|_\infty$ and the Bernstein inequality \cite[Th. 8.4.7]{DT}
 in the first integral  and $\Omega_\vphi^k(P,t)_{\vv}\le \C \|P \vv\|_\infty$ in the second one,   it is easy to deduce that
\[\|A^\rho P\|_{C_u}\leq \C \log m \|P\|_{C_{\vv}}. \]
Using the above inequality together with the Bernstein inequality \cite[Th. 8.4.7]{DT} we get
\begin{equation}\label{DAP-est}\|DA^{\rho}P\|_{C_{\vgdf}}\leq  \C  m \log m \| P\|_{C_{\vv}}.\end{equation}
Then, taking into account \eqref{Dap} and (see, for example, \cite{CaCriJuLu})
\begin{equation}\label{DAm1pm}(DA^{\rho})^{-1} p_m^{w}=\frac 1{m+1}p_m^{\rho},\end{equation}
the operator $DA^{\rho}:(\PP_{m-1}, \|\cdot\|_{C_{\vv}})\to (\PP_{m-1}, \|\cdot\|_{C_{\vgdf}})$ is continuous and invertible. Consequently, the operator $(DA^{\rho}+L_m^{w}K_m+L_m^{w}H):(\PP_{m-1}, \|\cdot\|_{C_{\vv}})\to (\PP_{m-1}, \|\cdot\|_{C_{\vgdf}})$ is continuous and invertible too and its inverse is bounded (see, for example, \cite[Theorem 3.4]{Kress}).

Now, for every $\mathbf{\theta}=(\theta_1, \ldots,\theta_m)^T$ there exists
$\mathbf{\eta}=(\eta_1,\ldots,\eta_m)^T$ such that $\mathbf{A}_m\mathbf{\theta}=\mathbf{\eta}$ iff
 $(DA^{\rho}+L_m^{w}K_m+L_m^{w}H)\tilde {\mathbf{\theta}}(y)= \tilde {\mathbf{\eta}}(y),$ where
\[\tilde{\mathbf{\theta}}(y)=\sum_{i=1}^m\psi_i^{\rho}(y)\theta_i, \quad \tilde
\eta(y)=\sum_{i=1}^m \psi^{w}_i(y)\eta_i,\]
with $\theta_i=(\vv\tilde \theta)(t_i)$ and $\eta_i=(u\varphi\tilde \eta)(x_i).$

Then, for every $\theta$ we have\\
\noindent $\|\mathbf{A}_m\theta\|_{\infty}= \|\eta\|_{\infty}=|\eta_\nu |=|(\tilde\eta v^{\ga,\be}\varphi)(x_\nu)|\leq
\|\tilde \eta\|_{C_{\vgdf}} $
\begin{eqnarray*}
\hphantom{H}&=&\|(DA^{\rho}+L_m^{w}K_m+L_m^{w}H)\tilde \theta\|_{C_{\vgdf}}\\
&\leq& \|DA^{\rho}\tilde \theta\|_{C_{\vgdf}}+\|L_m^{w}K_m\tilde \theta\|_{C_{\vgdf}}+\|L_m^{w}H\tilde \theta\|_{C_{\vgdf}} \\
&=:& N_1+N_2+N_3,
\end{eqnarray*}
where
$\displaystyle |\eta_\nu|=\max_{1\leq i\leq n}|\eta_i|.$ Using \eqref{DAP-est} and \eqref{lagal1al}, we get
\begin{eqnarray}\label{N1}N_1&\leq & \C m \log^2 m\|\theta\|_{\infty}.\end{eqnarray}
Moreover, by \eqref{lag1alal}, the definition \eqref{Km-def} of $K_m \tilde \theta$, Lemma \ref{Nevai}, the assumption \eqref{ky-cond} and \eqref{lagal1al}, we deduce
\begin{eqnarray*}
N_2&\leq &\C \log m \|K_m \tilde \theta\|_{C_{\vgdf}}\\
&\leq & \C \log m \|\tilde \theta\|_{C_{\vgdf}}\sup_{|y|\leq 1}(\vgdf)(y) \int_{-1}^1|L_m^{\rho}(k_y,x)|u^{-1}(x)dx \\
&\leq & \C \log m \|\tilde \theta\|_{C_{\vgdf}}\sup_{|y|\leq 1}(\vgdf)(y) \|k_y\|_\infty\\
&= & \C \log m \|\tilde \theta\|_{C_{\vgdf}}\sup_{|x|\leq 1}\|k_x \vgdf\|_\infty\\
&\leq & \C \log^2 m \|\theta\|_{C_{\vgdf}}.
\end{eqnarray*}
Finally, by \eqref{lag1alal}, the definition of $H \tilde \theta$, the assumption \eqref{h-cond1} and \eqref{lagal1al}, we have
\begin{eqnarray*}
N_3&\leq &\C \log m \|H \tilde \theta\|_{C_{\vgdf}}\\ &\leq &\C \log m \|\tilde \theta\|_{C_{\vgdf}}\sup_{|y|\leq 1}(\vgdf)(y) \int_{-1}^1|h(x,y)|u^{-1}(x)dx\\ &\leq &\C \log^2 m \|\theta\|_{\infty}.
\end{eqnarray*}
Summing up, we get
\begin{eqnarray}\label{normAA}
\|\mathbf{A}_m\|_{\infty}&\leq &\C m \log^2 m
\end{eqnarray}
Similarly proceeding, for every $\eta$, using \eqref{lag1alal}, we get
\begin{eqnarray*}
\|\mathbf{A}_m^{-1}\eta\|_{\infty} 
&\leq&   \C \left\|\left((DA^{\rho}+L_m^{w}K_m+L_m^{w}H)_{\ |\PP_{m-1}}\right)^{-1}\right\|_{C_{\vgdf}\to C_{\vv}}
\|\tilde \eta\|_{C_{\vgdf}} \\
&\leq&   \C \log m \left\|\left((DA^{\rho}+L_m^{w}K_m+L_m^{w}H)_{\ |\PP_{m-1}}\right)^{-1}\right\|_{C_{\vgdf}\to C_{\vv}}
\|\eta\|_{\infty}
\end{eqnarray*}
and, then,
\begin{eqnarray}\label{norminvA}
\|\mathbf{A}_m^{-1}\|_{\infty}&\leq &\C \log m \left\|\left((DA^{\rho}+L_m^{w}K_m+L_m^{w}H)_{\ |\PP_{m-1}}\right)^{-1}\right\|_{C_{\vgdf}\to C_{\vv}}.
\end{eqnarray}
Combining \eqref{normAA} and \eqref{norminvA}, the theorem follows.
\end{proof}

\subsection{Proof of Theorem \ref{teo-sigma}}

\begin{lemma}\label{M-compact-1} Under the assumptions $0\leq  \ga <  \frac 12, \, 0\leq \de < \frac{1}2,$
and if $\sigma\vphi\in Z_r$ with $r\geq 0$,
then the multiplying operator $M_{\sigma\vphi}:Z_{r}(\vgdf)\to Z_{r}(\vgdf)$ is continuous. Moreover, if $\sigma\vphi\in Z_r$ with $r\geq 1$, then  $M_{\sigma\vphi}:Z_{r}(\vgdf)\to Z_{r-1}(\vgd \vphi)$  is compact.
\end{lemma}
\begin{proof}
Since
\begin{equation}
\label{Msf-bound}
\|(M_{\sigma\vphi}f)\vgd\vphi\|_\infty=\|\sigma\vphi f\vgd\vphi\|_\infty\leq \|\sigma\vphi\|_\infty \|f\vgd\vphi\|_\infty
\end{equation}
and, by standard computation and the Favard inequality \eqref{Favard}, denoting by $\lfloor \!a\!\rfloor$ the greatest integer smaller or equal to $\!a\!>0,$ we have

{\small{\begin{equation}\label{EmM-est}
E_m(M_{\sigma\vphi}f)_{\vgd\vphi}\!\leq \! \|\sigma\vphi\|_\infty\! E_{\left\lfloor\frac m2\right\rfloor}(f)_{\vgd\vphi}\!+\! 2 \|f\vgd\vphi\|_\infty \! E_{\left\lfloor\frac m2\right\rfloor}(\sigma\vphi)
\!\leq \! \frac{\C}{m^r}\|f\|_{Z_r(\vgd\vphi)}\! \|\sigma\vphi\|_{Z_r},
\end{equation}}}
\hspace*{-1mm}it follows $\|M_{\sigma\vphi}f\|_{Z_r(\vgd\vphi)}\!=\!\|(M_{\sigma\vphi}f)\vgd\vphi\|_\infty\!+\! \sup_m \! m^r E_m(M_{\sigma\vphi}f)_{\!\vgd\vphi\!}\!<\! \C  \|f\|_{Z_r(\vgd\vphi)},$\\
i.e., the operator $M_{\sigma\vphi}:Z_{r}(\vgdf)\to Z_{r}(\vgdf)$ is continuous. Then $M_{\sigma\vphi}:Z_{r}(\vgdf)\to Z_{r-1}(\vgd \vphi)$ is compact, since $Z_{r}(\vgdf)$ is compactly imbedded into $Z_{r-1}(\vgd \vphi)$ (see \eqref{immZ}).
\end{proof}

\begin{proof}[Proof of Theorem \ref{teo-sigma}] The theorem follows by Corollary \ref{corolInv}, Lemmas \ref{K-compact-1}, \ref{H-compact-1} and \ref{M-compact-1} and the Fredholm alternative Theorem.
\end{proof}

\subsection{Proof of Theorems \ref{stability-convergence2} and \ref{condiz2}}

\begin{lemma}\label{ConvM} If, for some $s>0$, $0\leq \ga < \frac{1}4$ and $0\leq \de < \frac{1}4$, we have $\sigma\vphi\in Z_s$,
then, for every $1\leq r<s+1$,
\[\|(M_{\sigma\vphi}-L_m^{\vphi}M_{\sigma\vphi})f\|_{Z_{r-1}(\vgdf)}\leq \C \|f\|_{Z_s(\vgdf)}\|\sigma\vphi\|_{Z_{s}} \frac{\log m}{m^{s-r+1}}, \quad \C\neq \C(m,f,\sigma).\]
\end{lemma}
\begin{proof} By Lemma \ref{Lag1} for $\al=\frac 12$, under the assumptions on $\ga,\de$
\[\|\vgdf(M_{\sigma\vphi}-L_m^{\vphi}M_{\sigma\vphi})f\|_\infty\leq \C \log m E_m(M_{\sigma\vphi} f)_{\vgdf}.\]
Since $\sigma\vphi\in Z_s$, by \eqref{EmM-est} with $r=s$, we obtain
\[\|\vgdf(M_{\sigma\vphi}-L_m^{\vphi}M_{\sigma\vphi})f\|_\infty\leq \C \frac{\log m}{m^s} \|f\|_{Z_s(\vgd\vphi)} \|\sigma\vphi\|_{Z_s}.\]
Finally, by equivalence \eqref{norm2}, the lemma follows.
\end{proof}

\begin{proof}[Proof of Theorem \ref{stability-convergence2}]
Taking into account Lemma \ref{ConvM} the proof is similar to that of Theorem \ref{stability-convergence}.
\end{proof}
\begin{proof}[Proof of Theorem \ref{condiz2}] The proof is similar to that of Theorem \ref{condiz}
taking into account that, by \eqref{lag1alal}, \eqref{Msf-bound} and \eqref{lagal1al}, we get
\begin{eqnarray*}
\|L_m^{\vphi} M_{\sigma\vphi}\tilde \theta\|_{C_{\vgdf}} &\leq& \C \log m \|M_{\sigma\vphi}\tilde \theta\|_{C_{\vgdf}}\leq \C \log m \|\tilde \theta\|_{C_{\vgdf}}\\
&\leq & \C \log^2 m \|\theta\|_{\infty}.
\end{eqnarray*}
\end{proof}

\subsection{Proof of Proposition \ref{remark1}}
We prove (\ref{xi-value}) for  $h(x,y)=|x-y|^\mu$, since the other cases similarly follows.

\noindent Setting $\varphi_1(x)=\sqrt{1-|x|},$ by $\frac{\varphi(y)}{\sqrt{2}}\le \varphi_1(y)\le \varphi(y)$ it follows $\Omega_\varphi\sim\Omega_{\varphi_1}$ \cite{DT}. Assume at first $y\in [-1+4\tau^2,0]$ and consider the following decomposition
\begin{eqnarray*}
&&(u\varphi)(y)\int_{-1}^1 u^{-1}(x) |\Delta_{\tau\varphi_1(y)}k_x(y) |dx =(u\varphi)(y)\times \\
&\times &\left(\int_{-1}^{-1+\frac{1+y}2}+
\int_{-1+\frac{1+y}2}^{y-\tau \varphi_1(y)}+\int_{y-\tau \varphi_1(y)}^y+\int_y^{y+\tau\varphi_1(y)}+\int_{y+\tau\varphi_1(y)}^{y+\frac{1+y}2}+
\int_{y+\frac{1+y}2}^1\right)\\ && u^{-1}(x) |\Delta_{\tau\varphi_1(y)}k_x(y) |dx=:\sum_{k=1}^6 S_k(y).\end{eqnarray*}
Since  for $x<y-\frac {\tau} 2 \varphi_1(y)$ (\cite[(13.5.3)]{LMR})
\begin{equation}\label{uno}|\Delta_{\tau\varphi_1(y)}k_x(y)|\le \tau\varphi_1(y)\left(y-\frac \tau 2 \varphi_1(y)-x\right)^{\mu-1}, \end{equation}
and  for $y\in [-1+4\tau^2,0],$ by $\mu-1<0$,
$\left(y-\frac \tau 2 \varphi_1(y)-x\right)^{\mu-1}\le $ $\left( \frac{1+y}4\right)^{\mu-1},$  we have
\begin{equation*}
S_1(y)\le \C \tau (1+y)^{\de+\mu}\int_{-1}^{-1+\frac{1+y}2}(1+x)^{-\de}dx\le
\C \tau(1+y)^{\mu+1}\le \C \tau,\end{equation*}
being  $\mu+1>0$. By (\ref{uno}) again
\begin{eqnarray*}S_2(y)\le \C \tau (1+y)^{\de+1}\int_{-1+\frac{1+y}2}^{y-\tau \varphi_1(y)}\left(y-\frac \tau 2 \varphi_1(y)-x\right)^{\mu-1}(1+x)^{-\de}dx.
\end{eqnarray*}
Then, by $(1+x)\ge \frac{1+y}2$ and  setting $y-x=u\sqrt{1+y},$ it follows
\begin{eqnarray*}S_2(y)\le \! \C \tau (1+y)^{1+\frac\mu 2} \int_{\tau}^{\frac{\sqrt{1+y}}2}\!\!\!\left(\!u-\!\frac {\tau} 2\!\right)^{\mu-1}\!\!\!dx\!=\!\C \tau \frac{(1+y)^{1+\frac\mu 2}}\mu\! \!\left[\!\left (\!\frac{\sqrt{1+y}}2\!-\!\frac {\tau} 2\!\right)^{\mu}\!-\!\tau^{\mu}\!\right]\end{eqnarray*}
and, using $\frac{\sqrt{1+y}}2\ge \tau$, we can conclude
$
S_2(y)\le \C \tau^{1+\mu}.$ 
Similar  estimates hold for $S_5$ and $S_6$.
To estimate $S_3$ we use
$|\Delta_{\tau\varphi_1(y)}k_x(y)|\le \left|y-\frac {\tau} 2 \varphi_1(y)-x\right|^{\mu}$
and by $1+x\sim 1+y$,
\begin{eqnarray*}S_3(y)\le \C (1+y)^{\frac 1 2}\int_{y-\tau \varphi_1(y)}^y \left|y-\frac {\tau} 2 \varphi_1(y)-x\right|^{\mu}dx\end{eqnarray*} and by the change of variable $y-x-\frac {\tau} 2 \sqrt{1+y}=\theta$ it follows
\begin{eqnarray*}S_3(y)\le \C (1+y)^{\frac 1 2}\int_{-\frac \tau 2 \sqrt{1+y}}^{\frac \tau 2 \sqrt{1+y}}|\theta|^\mu d\theta =
\C \tau^{\mu+1} (1+y)^{\frac 3 2+\mu}\le \C \tau^{\mu+1} .\end{eqnarray*}
Similarly we estimate  $S_4$ by using $|\Delta_{\tau\varphi_1(y)}k_x(y)|\le \left|y+\frac \tau 2 \varphi_1(y)-x\right|^{\mu},$
and the lemma is proved for $y\in [-1+4\tau^2,0]$. We omit the proof in the case $y\in [0,1-4\tau^2]$, since it follows by similar arguments.

\section{Conclusions}
In this paper we have proposed a numerical scheme based on Lagrange projection  for solving integral equations of the kinds  (\ref{eqSigma0}) and (\ref{eqSigma}).  The approximate solution has been obtained by solving a system of algebraic equations, whose conditioning has been studied.
Stability and convergence have been proved, giving estimates of the errors in Zygmund norm. We have illustrated various aspects of the theory by means of some examples,  evaluating  the efficiency of the proposed scheme from different points of view. In the first test we have   compared our results  with those reached by the  procedure proposed in \cite{Calio}, by showing that our method faster converges. Examples 2 and 3 have been devoted especially to test the agreement of the predicted orders of convergence with the numerical $EOCs$, choosing for this goal functions of different smoothness. In both examples numerical evidence shows also that the condition numbers  of the linear systems diverge at most like  $m \log^3 m$. This fact encourages us to believe that the norms in \eqref{cond1} and \eqref{cond2} do not increase w.r.t. $m$. Moreover, in Example 2 we have compared the condition numbers of the linear systems of our procedures  with those of the procedure  in \cite{CaCriJuLu}, showing the  substantial different behaviors between them (see Table \ref{tab:b22}). The better conditioning  in our procedure is ascribable to the choice of the basis to represent the Lagrange polynomials (see  \cite{La,DBSIam}).

Finally we have considered the application of our method in solving  the  Prandtl's equation governing the circulation air flow   along the contour of a plane wing profile,
for two different wing-shapes. Also in these cases we have shown that our experimental results are more accurate  than  those obtained in \cite{dragos1}, \cite{dragos2}, highlighting  once again the efficiency of our  approach.

\subsection*{Acknowledgment} This research was supported by University of Basilicata (local funds) and by GNCS Project 2019 ``Discretizzazione di misure, approssimazione di operatori integrali ed applicazioni''.


\bibliography{bibDeBonisOccorsio}

\begin{thebibliography}{10}
\expandafter\ifx\csname url\endcsname\relax
  \def\url#1{\texttt{#1}}\fi
\expandafter\ifx\csname urlprefix\endcsname\relax\def\urlprefix{URL }\fi
\expandafter\ifx\csname href\endcsname\relax
  \def\href#1#2{#2} \def\path#1{#1}\fi

\bibitem{vainikko}
I.~Lifanov, L.~Poltavskii, G.~Vainikko, Hypersingular Integral Equations and
  their Applications, Chapman \& Hall \/ CRC, 2003.

\bibitem{muske}
N.~I. Muskhelishvili, Singular Integral Equations. Boundary Problems of
  Function Theory and Their Application to Mathematical Physics, Dover Books on
  Mathematics, Dover Publications, 1953.

\bibitem{mkhi}
S.~M. Mkhitaryan, M.~S. Mkrtchyan, E.~G. Kanetsyan, On a method for solving
  {P}randtl's integro-differential equation applied to problems of continuum
  mechanics using polynomial approximations, Z. Angew. Math. Mech. 97~(6)
  (2017) 639--654.

\bibitem{aruty}
N.~K. Arutyunyan, S.~M. Mkhitaryan, Some contact problems for a semi-plane with
  elastic stiffeners, Trends in Elasticity and Thermoelasticity,
  Wolters-Noordhoff Publ., Groningen, 1971.

\bibitem{smirnov}
V.~V. Sil'vestrov, A.~V. Smirnov, The {P}randtl's integrodifferential equation
  and the contact problem for a piecewise homogeneous plate, J. Appl. Math.
  Mech. 74~(6) (2010) 679--691.

\bibitem{MoPe}
G.~Monegato, V.~Pennacchietti, Quadrature rules for {P}randtl's integral
  equation, Computing 37 (1986) 31--42.

\bibitem{Pra}
L.~Prandtl, The Mechanics of Viscous Fluids, Vol. III of Aerodynamic Theory,
  Springer, Berlin, 1935.

\bibitem{dragos1}
L.~Dragos, Integration of {P}randtl's equation with the aid of quadrature
  formulae of {G}auss type, Quarterly of Applied Mathematics LII (1994) 23--29.

\bibitem{dragos2}
L.~Dragos, A collocation method for the integration of {P}randtl's equation,
  ZAMM-Z Angew Math Mech 74~(7) (1994) 289--290.

\bibitem{Gol}
V.~V. Golubev, Lectures on the wing theory, Gostechizdat, Moscow–Leningrad,
  1949.

\bibitem{vekua}
I.~N. Vekua, On the integro-differential equation of {P}randtl, Prikl. Mat.
  Mekh. 9~(2) (1945) 143--150.

\bibitem{kalandiya}
A.~I. Kalandiya, Mathematical Methods of Two-Dimensional Elasticity, Nauka,
  Moscow, Leningrad, 1949.

\bibitem{golberg}
M.~A. Golberg, The convergence of several algorithms for solving integral
  equations with finite part integrals. {II}, Appl. Math. Comput. 21 (1987)
  283--293.

\bibitem{CaCriJu}
M.~R. Capobianco, G.~Criscuolo, P.~Junghanns, A fast algorithm for {P}randtl's
  integro-differential equation, J. Comp. Appl. Math. 77 (1997) 103--128.

\bibitem{Calio}
F.~Cali\'o, E.~Marchetti, On an algorithm for the solution of generalized
  {P}randtl equations, Numer. Algor. 28 (2001) 3--10.

\bibitem{CaCriJuLu}
M.~R. Capobianco, G.~Criscuolo, P.~Junghanns, U.~Luther, Uniform convergence of
  the collocation method for {P}randtl's integro-differential equation, ANZIAM
  J. 42 (2000) 151--168.

\bibitem{kutt}
H.~Kutt, On the numerical evaluation of finite-part integrals involving an
  algebraic singularity, PhD thesis, University of Stellenbosch, 1975.

\bibitem{monegato}
G.~Monegato, Numerical evaluation of hypersingular integrals, J. Comp. Appl.
  Math. 50 (1994) 9--31.

\bibitem{DBOLagos}
M.~De~Bonis, D.~Occorsio, On the simultaneous approximation of a {H}ilbert
  transform and its derivatives on the real semiaxis, Appl. Numer. Math. 114
  (2017) 132--153.

\bibitem{miklin}
S.~G. Mikhlin, S.~Pr\"ossdorf, Singular {I}ntegral {O}perators, (translated
  from German), Akademie-Verlag, Berlin, 1986.

\bibitem{PSbook}
S.~Pr\"ossdorf, B.~Silbermann, Numerical analysis for integral and related
  operator equations, Akademie-Verlag, Berlin, 1991.

\bibitem{mastromilo}
G.~Mastroianni, G.~V. Milovanovi\'c, Interpolation Processes Basic Theory and
  Applications, Springer Monographs in Mathematics, Springer-Verlag, Berlin,
  Heidelberg, 2009.

\bibitem{DT}
Z.~Ditzian, W.~Totik, Moduli of smoothness, SCMG Springer-Verlag, New York
  Berlin Heidelberg London Paris Tokyo, 1987.

\bibitem{dbmastroparma}
M.~C. De~Bonis, G.~Mastroianni, Direct methods for {CSIE} in weighted {Z}ygmund
  spaces with uniform norm, Riv. Math. Univ. Parma 2~(1) (2011) 29--55.

\bibitem{La}
C.~Laurita, Condition numbers for singular integral equations in weighted $l^2$
  spaces, J. Comp. Appl. Math. 116 (2000) 23--40.

\bibitem{DBSIam}
M.~C. De~Bonis, G.~Mastroianni, Projection methods and condition numbers in
  uniform norm for {F}redholm and {C}auchy singular integral equations, SIAM J.
  Numer. Anal. 44~(4) (2006) 1351--1374.

\bibitem{MR}
G.~Mastroianni, M.~G. Russo, Lagrange interpolation in some weighted uniform
  spaces, Facta Univ. Ser. Math. Inform. 12 (1997) 185--201.

\bibitem{Nevai}
P.~Nevai, Mean convergence of {L}agrange interpolation. {III}, Trans. Amer.
  Math. Soc. 282~(2) (1984) 669--698.

\bibitem{Szego}
G.~Szeg\H{o}, Orthogonal {P}olynomials, Vol. 23 4th ed. of Amer. Math. Soc.
  Colloq. Publ., Math. Soc., Providence, 1975.

\bibitem{MR2}
G.~Mastroianni, M.~G. Russo, Lagrange interpolation in weighted {B}esov spaces,
  Constr. Approx. 15 (1999) 257–289.

\bibitem{F}
L.~Fermo, Embedding theorems for functions with inner singularities, Acta
  Scientiarum Mathematicarum (Szeged) 75 (2009) 547--573.

\bibitem{MRT}
G.~Mastroianni, M.~G. Russo, W.~Themistoclakis, Numerical methods for {C}auchy
  singular integral equations in spaces of weighted continuous functions,
  Recent advances in operator theory and its applications, Oper. Theory Adv.
  Appl., Birkh\"auser, Basel 160 (2005) 311--336.

\bibitem{Timan}
A.~F. Timan, Theory of {A}pproximation of {F}unctions of a {R}eal {V}ariable,
  Dover Pubblications, Inc. New York, 1994.

\bibitem{russoPhD}
M.~G. Russo, Proiettori in spazi funzionali ed equazioni integrali, PhD thesis.

\bibitem{Atk}
K.~E. Atkinson, The Numerical Solution of Integral Equations of the second
  kind, Cambridge Monographs on Applied and Computational Mathematics,
  Cambridge University Press, 1997.

\bibitem{D}
M.~C. De~Bonis, Remarks on two integral operators and numerical methods for
  {CSIE}, J. Comp. Appl. Math. 260 (2014) 117--134.

\bibitem{JuLu}
P.~Junghanns, U.~Luther, Cauchy singular integral equations in spaces of
  continuous functions and methods for their numerical solution, J. Comp. Appl.
  Math. 77~(1-2) (1997) 201--237.

\bibitem{MT2}
G.~Mastroianni, W.~Themistoclakis, A numerical method for the generalized
  airfoil equation based on the de la {V}allée {P}oussin interpolation, J.
  Comp. Appl. Math. 180 (2005) 71--105.

\bibitem{Kress}
R.~Kress, Linear Integral Equations, Vol.~82 of Applied. Mathematical Sciences,
  Springer-Verlag, Berlin etc., 1989.

\bibitem{russoLuther}
M.~G. Russo, U.~Luther, Boundedness of the {H}ilbert transformation in some
  weigheted {B}esov type spaces, Integr. {E}qu. {O}per. {T}heory 36 (2000)
  220--240.

\bibitem{LMR}
C.~Laurita, G.~Mastroianni, M.~G. Russo, Revisiting {CSIE} in ${L}^2$:
  condition numbers and inverse theorems, in: Integral and Integrodifferential
  Equations, Gordon and Breach, Amsterdam, 2000, Ch.~2, pp. 159--184.

\end{thebibliography}

{\small{\it Maria Carmela De Bonis, {Department of Mathematics, Computer Science and Economics, University of Basilicata, Via dell'Ateneo Lucano 10, 85100 Potenza, ITALY.} mariacarmela.debonis@unibas.it.

\vspace{0.5cm}

Donatella Occorsio {Department of Mathematics, Computer Science and Economics, University of Basilicata, Via dell'Ateneo
Lucano 10, 85100 Potenza, ITALY. \\ donatella.occorsio@unibas.it.}}}

\end{document}